\documentclass[11pt]{article}

\usepackage{todonotes}

\usepackage[top=3.5cm,bottom=3.5cm,left=3.5cm,right=3cm]{geometry}
\pdfpagewidth\paperwidth
\pdfpageheight\paperheight
\usepackage[english,italian]{babel}
\usepackage[T1]{fontenc}
\usepackage[latin1]{inputenc}
\usepackage{lmodern}
\usepackage{amsfonts}
\usepackage{cancel}
\usepackage{amsmath}
\usepackage{amsthm}
\usepackage{amssymb}
\usepackage{graphicx}
\usepackage{sidecap}
\usepackage{caption}
\usepackage{subfig}
\usepackage{wrapfig}
\usepackage{psfrag}
\usepackage{mathrsfs}
\usepackage{tikz}
\usepackage{multicol}
\usepackage{pgfplots}
\usepackage{hyperref,enumitem}

\usepackage{color}

\newtheorem{theorem}{Theorem}
\newtheorem{lemma}[theorem]{Lemma}
\newtheorem{proposition}[theorem]{Proposition}

\theoremstyle{definition}
\newtheorem{definition}[theorem]{Definition}

\theoremstyle{remark}

\newtheorem*{remark*}{Remark}

\newcommand{\D}{{\rm d}}
\newcommand{\FBN}{\mathbb N}
\newcommand{\FBR}{\mathbb R}
\newcommand{\FBC}{\mathbb C}

\newcommand{\FBcD}{{\mathcal D}}
\newcommand{\FBcE}{{\mathcal E}}
\newcommand{\FBcF}{{\mathcal F}}
\newcommand{\FBcO}{{\mathcal O}}

\newcommand{\FBcS}{{\mathcal S}}

\newcommand{\email}[1]{\protect\href{mailto:#1}{#1}}

\begin{document}
\selectlanguage{english}

\title{Continuity properties of the shearlet  transform and the shearlet synthesis
operator on the Lizorkin type spaces}
% of test functions and distributions}

%\titlerunning{The Shearlet transform and Lizorkin spaces}
% Use \titlerunning{<name>} for an abbreviated version of
\author{Francesca Bartolucci \thanks{Department of Mathematics, ETH Zurich, Raemistrasse 101, 8092 Zurich, Switzerland (\email{francesca.bartolucci@sam.math.ethz.ch}).}
\and  Stevan Pilipovi\'c \thanks{Department of Mathematics and Informatics, Faculty of Sciences, University of Novi Sad, Trg Dositeja Obradovi\' ca 4, 21000 Novi Sad, Serbia (\email{stevan.pilipovic@gmail.com}, \email{nenad.teofanov@dmi.uns.ac.rs}).}
\and Nenad Teofanov \footnotemark[2]}
%\author{Francesca Bartolucci, Stevan Pilipovi\'c and Nenad Teofanov}
%% Use \authorrunning{Short Title} for an abbreviated version of
%% your contribution title if the original one is too long
%\institute{Francesca Bartolucci \at ETH Zurich, Department of Mathematics, Raemistrasse 101, 8092 Zurich, Switzerland, \email{francesca.bartolucci@sam.math.ethz.ch}
%\and Stevan Pilipovi\'c \at University of Novi Sad, Faculty of Sciences, Department of Mathematics and Informatics, Trg Dositeja
%Obradovi\' ca 4, 21000 Novi Sad, Serbia \email{stevan.pilipovic@gmail.com}
%\and Nenad Teofanov \at University of Novi Sad, Faculty of Sciences, Department of Mathematics and Informatics, Trg Dositeja
%Obradovi\' ca 4, 21000 Novi Sad, Serbia \email{nenad.teofanov@dmi.uns.ac.rs}}
\maketitle

\abstract{We develop a distributional framework for the shearlet transform $\FBcS_{\psi}\colon\FBcS_0(\FBR^2)\to\FBcS(\mathbb{S})$ and the shearlet synthesis operator
$\FBcS^t_{\psi}\colon\FBcS(\mathbb{S})\to\FBcS_0(\FBR^2)$, where $\FBcS_0(\FBR^2)$ is the Lizorkin test function space and $\FBcS(\mathbb{S})$ is the space of highly localized test functions on
the standard shearlet group $\mathbb{S}$. These spaces and their duals $\mathcal{S}_0^\prime (\mathbb R^2),\, \mathcal{S}^\prime (\mathbb{S})$ are called Lizorkin type spaces of test functions and distributions. We analyze the continuity properties of these transforms
when the admissible vector $\psi$ belongs to $\FBcS_0(\FBR^2)$. Then, we define
the shearlet transform and the shearlet synthesis operator of Lizorkin type distributions as
transpose mappings of the shearlet synthesis operator and the shearlet transform, respectively.
They yield continuous mappings from  $\mathcal{S}_0^\prime (\mathbb R^2)$ to $\mathcal{S}^\prime (\mathbb{S})$ and from
$\mathcal{S}^\prime (\mathbb S)$ to $\mathcal{S}_0^\prime (\mathbb{R}^2)$. Furthermore, we show
the consistency of our definition with the shearlet transform  defined by direct evaluation of a distribution on the shearlets. The same can be done for the shearlet synthesis operator.
Finally, we give a reconstruction formula for Lizorkin type distributions, from which follows that
the action of such generalized functions can be written as an absolutely convergent integral over the standard shearlet group.
}\\\\


\vspace{2mm}

\section{Introduction}
\label{FBsec:introduction}

In this paper we provide a sound mathematical background for the shearlet transform and the corresponding shearlet synthesis operator when acting on  test function and distribution spaces.
This is done by studying the continuity properties of these two integral transforms on the Lizorkin type spaces of test functions and, by duality, on the corresponding distribution spaces.
Our results are consistent with the ones given in \cite{dahlikeetal, FBgr12, FBkula09} for $L^2$-functions. Actually, concerning the continuity results, we extend the corresponding ones of quoted papers.

The theory of shearlets emerged from the research activities aimed to create a new generation of analysis and processing
tools for massive and higher dimensional data which could go beyond traditional Fourier and wavelet systems, \cite{FBcado99,FBkula12}.
In signal analysis it is customary to transform a signal, modeled as an element of
a Hilbert space $\mathcal{H}$, into a new function or distribution over a convenient parameter space
in order to extract the most relevant information in the most efficient way, \cite{FBddgl15}.
For example, the wavefront set carries both the information on the location and the geometry of the singularity set of a signal and it can be resolved by means of the shearlet transform,
\cite{FBgr11, FBgr12}.

\par

In a certain sense, shearlets behave for high-dimensional signals as wavelets do for one-dimensional ones.
We refer to   \cite{FBbardemadeviodo}
for the link between these two classical transforms in the framework of square-integrable group representations theory, which offers a unified approach to treat different relevant
representations in signal analysis when the parameter space forms a group. In such a case, signals are mapped via the so-called voice transform in a new function on the group, which is square-integrable with respect to a Haar measure under some technical assumptions.

\par

An important issue in harmonic analysis is the extension of an integral transform from
a Hilbert space to a more general framework of generalized function spaces, and there are two
basic approaches to deal with this problem.
The first approach is to  define the generalized transform through the action of a distribution on a family of test functions in a suitable test function space.
In this respect, the coorbit space theory introduced by Feichtinger and Gr\"ochenig in \cite{feichtingergrochenig89,feichtgroche89}
is a powerful tool which applies when the integral transform is the voice transform associated to a square-integrable representation of a locally compact group.
This is the case of the shearlet transform, and we refer to \cite{dahlikeetal} for its extension based
on the coorbit space theory, which yields completely new regularity spaces, the shearlet coorbit spaces, and the corresponding dual spaces of distributions.
The second approach, which we employ in the current paper, consists in defining a formal transpose of the integral transform,
and then the extended transform between pairs of dual spaces by transposition.
The equivalence between these two approaches can be showed by comparing them on the intersection of their domains.

\par

We follow the second approach based on the duality theory related to locally convex spaces, i.e. the distribution theory introduced by Schwartz in \cite{FBLS1951}. Moreover, we  give the necessary comparison with the direct approach.
Our work is motivated by the lack of a complete and rigorous distributional framework for the shearlet transform in the literature.
On the contrary, related integral transforms, such as the Radon and the wavelet transforms, are well understood and deeply studied in the context of test function and distribution spaces
in e.g. \cite{FBhelgason99,FBhertle83,FBhol1995, FBkpsv2014, FBPRTV}.

\par

It turns out that the Lizorkin type spaces of test functions $\mathcal{S}_0 (\mathbb{R}^{2})$ and $\mathcal{S}(\mathbb S)$ are the natural domain and range space for the shearlet transform. The same holds for the shearlet synthesis operator but in the opposite direction.
Recall, the Lizorkin space $\mathcal{S}_0 (\mathbb{R}^{d})$ consists of smooth,  rapidly decreasing functions with vanishing moments of any order. 
%The extensions and the analysis within  distribution spaces $\mathcal{S}^\prime (\mathbb S)$ and $\mathcal{S}^\prime_0 (\mathbb{R}^{2})$ are done through the transposed mappings.
Kostadinova et al. \cite{FBkpsv2014} showed that the domain of the ridgelet transform can be enlarged to its dual space $\FBcS'_0(\FBR^2)$, known as the space of Lizorkin distributions.
The proof exploits the close connection between the Radon, the ridgelet and the wavelet transforms.
This result, in combination with \cite{FBbardemadeviodo}, yields a relation formula between the shearlet transform and the ridgelet transform,
as it is shown in \cite{FBbarpilnen}.

\par

The vanishing moments condition for the functions in the test space  is not surprising. Indeed, as pointed out in the wavelet  \cite{FBmallat09} and  the shearlet analysis \cite{FBgr11}, as well as  in the recent
study of the Taylorlet transform \cite{FBfinkkahler19}, vanishing moments are crucial in order to measure the local regularity and to detect anisotropic structures of a signal.
Furthermore, we notice that the rectified linear units (ReLUs), which are important examples of
unbounded activation functions in the context of deep learning neural networks, belong to the space of Lizorkin distributions, see
\cite{FBsomu17} for details.

\par

Our main results are continuity theorems for the shearlet transform and the  shearlet synthesis operator on appropriate test function spaces.
Since the image of a signal under the shearlet transform is a function over the shearlet group $\mathbb{S} = \FBR^2\times\FBR\times\FBR^{\times}$,
and our aim is to analyze its regularity and decay properties with respect to all parameters, the corresponding topology is given by a family of norms
determined by eight indices. We prove that the shearlet transform $\FBcS_\psi\colon \mathcal{S}_0(\FBR^2)\to \mathcal{S} (\mathbb{S}) $ and the corresponding synthesis operator $\FBcS_\psi^t\colon \mathcal{S}(\mathbb{S})\to \mathcal{S} (\FBR^2)$ are continuous mappings, where $\mathcal{S}(\mathbb{S})$
is  endowed with the above mentioned topology.
Then, we use these continuity results to extend the shearlet transform and the shearlet synthesis operator to the space of Lizorkin type distributions $\mathcal{S}^\prime_0 (\mathbb{R}^{2})$ and $\mathcal{S}^\prime (\mathbb S)$ by the duality approach and following several ideas in \cite{FBkpsv2014}.
We show the continuity properties of the shearlet transform when acting on $\mathcal{S}^\prime_0 (\mathbb{R}^{2})$.  Observe that many important Schwartz distribution spaces, such as $\FBcE ' (\mathbb{R}^{d})$,
$\FBcO _C' (\mathbb{R}^{d})$, $ L^p  (\mathbb{R}^{d})$ and $\FBcD _{L^1} ' (\mathbb{R}^{d})$ are embedded into
the space of Lizorkin distibutions  $\mathcal{S}_0 ' (\mathbb{R}^{d})$, see e.g. \cite{FBkpsv2014} for details.

\par

We complete our analysis by showing that our definition of the shearlet transform of distributions extends the ones considered so far in e.g. \cite{FBkula12, FBgr11},
and that it is consistent with those for test functions. Moreover, it follows from our analysis that our duality approach is equivalent to the one based on the coorbit space theory presented in \cite{dahlikeetal}.

\par

\subsection{Notation}
\label{FBsubsec:1.1}
We briefly introduce the $d$-dimensional notation. We set $\mathbb{N} = \{ 0,1,2, \dots \}$, $ \mathbb{Z}_+$ denotes the set of positive integers, $\FBR_+=(0,+\infty)$ and $\FBR^{\times}=\FBR\setminus\{0\}$.
When $ x, y \in \mathbb{R}^d $, $ xy = x_1 y_1 + x_2 y_2 + \dots + x_d y_d, $
$|x|$ denotes the Euclidean norm, $ \langle x \rangle = (1+|x|^2)^{1/2}$,
$ x^{m} = x_1^{m_1} \dots x_d^{m_d}$, and $\partial^{m}=\partial_x^{m} =
\partial_{x_1}^{m_1} \dots\partial_{x_d}^{m_d}$,  $ m \in \mathbb{N}^d $.
We write also $\varphi^{(m)}=\partial^{m}\varphi$, $m\in\FBN^d$.
By a slight abuse of notation, the length of a multi-index
$ m \in \mathbb{N}^d $ is denoted by $ |m| = m_1 + \dots + m_d $ and
the meaning of $|\cdot| $ shall be clear from the context.
For every $b\in\FBR^d$, the translation operator acts on a
function $f:\FBR^d\to \FBC$ as
$
T_bf(x)=f(x-b)
$
and the dilation operator $D_a\colon L^p(\FBR^d)\to L^p(\FBR^d)$ is defined by $D_a f(x)=|a|^{-\frac{1}{2}}f(x/a)$ for every $a\in\FBR^{\times}$.
We write $ A\lesssim B $ when $ A \leq C \cdot B $ for
some positive constant $C$.

\par

As usual, $L^p(\FBR^d)$, $p\in[1,+\infty)$, is the Banach space of  $p$-integrable functions
$f\colon\FBR^d\rightarrow\FBC$ with respect to the Lebesgue measure $\D x$ and, if $p=2$, the corresponding scalar product and norm are
$\langle\cdot,\cdot\rangle$ and $\|\cdot\|$, respectively. The space $ L^\infty(\FBR^d)$ consists of essentially bounded functions.
The Fourier transform $\mathcal F $ is given by
\begin{equation*}
\mathcal F f({\xi})= \int_{\FBR^d} f(x) e^{-2\pi i\,
  {\xi}x } \D{x},\qquad f\in L^1(\FBR^d),
\end{equation*}
and it extends to $L^2(\FBR^d)$ in the usual way.

The dual pairing between a test function space $ {\mathcal A}$ and its dual space of distributions
${\mathcal A'}$ is denoted by ${_{\mathcal A'}(\: \cdot\:, \:\cdot \:)_{\mathcal A} }$
and we provide all distribution spaces with the strong dual topologies. For simplicity in the notation, we also use
dual pairings without explicitly stating the spaces ${\mathcal A}, {\mathcal A'}$ and they will be clear from the context.
\par
The Schwartz space of rapidly decreasing smooth test functions is denoted by $\mathcal{S}(\mathbb{R}^{d})$
and $\mathcal{S}'(\mathbb{R}^{d})$ denotes its dual space of tempered distributions. Concerning the family of norms
which defines a projective limit topology  on $ \mathcal{S} (\FBR^d)$, we make the choice
\begin{equation*}
\rho_\nu (\varphi ) = \sup_{x \in \FBR^d, |m|\leq \nu} \langle x \rangle ^\nu |\partial^{m} \varphi(x) |,
\end{equation*}
for every $ \nu \in \FBN$ and $\varphi \in  \mathcal{S} (\FBR^d)$.

\par

If $G$ is a locally compact group, we denote by $L^2(G)$ the Hilbert
space of square-integrable functions with respect to a left Haar
measure on $G$. If $A\in M_{d}(\FBR)$, the vector space of square $d\times d$ matrices with real entries, $^t\! A$ denotes its transpose and we denote the (real) general linear group of size $d\times d$ by ${\rm GL}(d,\FBR)$.

\section{Preliminaries}
In this section we first introduce the shearlet transform. We refer to \cite{FBfolland16}
as a classical reference for the theory of group representations of locally compact groups and to \cite{FBkula12} for a complete overview of shearlet analysis.
Then, we introduce the Lizorkin space of test functions, a closed subspace of $\mathcal{S}(\mathbb{R}^{d})$ which plays a crucial role in our analysis, and the space $ \mathcal{S} (\mathbb{S}) $, which contains the range of the shearlet transform when acting on the Lizorkin space.
As already mentioned, both spaces and their duals are called Lizorkin type spaces of test functions and distributions.

\subsection{The Shearlet transform}
\label{FBsec:3}
Although we will give the definitions of multi-dimensional Lizorkin test and distribution type spaces, when dealing with the shearlet transform and the shearlet synthesis operator we consider $\mathbb R^2$ and the standard shearlet group  $\mathbb{S}=\FBR^2\times\FBR\times\FBR^{\times}$ endowed with the group operation
\begin{equation*}
(b,s,a)(b',s',a')=(b+N_sA_ab',s+|a|^{1/2}s',aa'),
\end{equation*}
where
%the matrices $N_s$ and $A_a$ are defined as
\begin{equation*}
N_s=\left[\begin{matrix}1 & -s\\ 0 & 1\end{matrix}\right],\qquad A_a=a\left[\begin{matrix}1 & 0\\ 0 & |a|^{-1/2}\end{matrix}\right], \;\;\;
s\in\FBR, a\in\FBR^\times.
\end{equation*}
A left Haar measure on $\mathbb{S}$ is given by
\begin{equation*}
{\rm d}\mu(b,s,a)=|a|^{-3}{\rm d}b{\rm d}s{\rm d}a,
\end{equation*}
where ${\rm d}b$, ${\rm d}s$ and ${\rm d}a$ are the Lebesgue measures on $\FBR^2$, $\FBR$ and $\FBR^{\times}$, respectively.
The group $\mathbb{S}$ acts on $L^2(\FBR^2)$ via the square-integrable representation
\begin{equation}\label{eq:shearletrappr}
\pi_{b,s,a}f(x)=|a|^{-3/4}f(A_a^{-1}N_s^{-1}(x-b)),
\end{equation}
or, equivalently, in the frequency domain
\begin{equation} \label{FBshearfreq}
\FBcF \pi_{b,s,a}f(\xi)=|a|^{3/4} e^{-2\pi i b\xi}\FBcF f(A_a{^t\!N_s}\xi).
\end{equation}
We denote by $C(\mathbb{S})$ the space of continuous functions on $\mathbb{S}$ and by $L^{\infty}(\mathbb{S})$ the space of essentially bounded functions on $\mathbb{S}$ with respect to the Haar measure ${\rm d} \mu$.
\begin{definition} \label{FBdefsheartransf}
Let $\psi\in L^2(\FBR^2)$. The shearlet transform associated to $\psi$ is the map $\FBcS_{\psi}\colon L^2(\FBR^2)\to C(\mathbb{S})\cap L^{\infty}(\mathbb{S})$ defined by
\begin{equation*}
\FBcS_{\psi}f(b,s,a)=\langle f,\pi_{b,s,a}\psi\rangle
= |a|^{-\frac{3}{4}}\int_{\FBR^2}f(x)\overline{\psi(A_{a}^{-1} N_{s}^{-1}(x-b))}\D x,
\end{equation*}
for every $(b, s,a) \in \mathbb{S}$.
\end{definition}

It is well-known that the shearlet transform $\FBcS_{\psi}$
is a non-trivial multiple of an isometry from $L^2(\FBR^2)$ into $L^2(\mathbb{S},{\rm d}\mu)$ provided that $\psi\in L^2(\FBR^2)$ satisfies the admissibility condition
\begin{equation} \label{FBeqn:admvect}
0<C_{\psi}=\int_{\FBR^2}\frac{|\FBcF\psi(\xi)|^2}{|\xi_1|^2}\D\xi<+\infty,
\end{equation}
where $\xi=(\xi_1,\xi_2)\in\FBR^2$, or equivalently
\begin{equation}\label{FBeqn:admvectequivalent}
C_{\psi}  = \int_{\FBR^\times}\int_{\FBR}|\FBcF\psi(A_a{^{t}\!N_s}\xi)|^2\D s\frac{\D a}{|a|^{\frac{3}{2}}},\qquad \text{for a.e.} \; \xi\in\FBR^2 \setminus \{0\},
\end{equation}
see e.g. \cite{FBdahlke2008}. Functions  $\psi\in L^2(\FBR^2)$ which satisfy \eqref{FBeqn:admvect} are called admissible shearlets.
Furthermore, if $\psi\in L^2(\FBR^2)$ is an admissible shearlet then for every $f\in L^2(\FBR^2)$ we have the reconstruction formula
\begin{equation}\label{FBreconstructionformulashearlet}
 f = \frac{1}{C_{\psi}}  \int_{\FBR^{\times}}\int_{\FBR}\int_{\FBR^2}
 \FBcS_{\psi}f(b,s,a) \,     \pi_{b,s,a}\psi\ \frac{\D b\D s\D a}{|a|^3},
\end{equation}
where the equality holds in the weak-sense, i.e.
\[
\langle f,g\rangle = \frac{1}{C_{\psi}}  \int_{\FBR^{\times}}\int_{\FBR}\int_{\FBR^2}
\FBcS_{\psi}f(b,s,a) \,     \langle\pi_{b,s,a}\psi,g\rangle\ \frac{\D b\D s\D a}{|a|^3},\qquad g\in L^2(\FBR^2).
\]
 We can also consider a reconstruction formula when
the analyzing and the synthesis vectors do not coincide. Precisely, for every pair of functions $\psi,\phi\in L^2(\FBR^2)$ which
satisfy the admissible condition
\begin{equation} \label{FBeqn:admvect2}
0<C_{\psi,\phi}=\int_{\FBR^2}\frac{\overline{\FBcF\psi(\xi)}\FBcF\phi(\xi)}{|\xi_1|^2}\D\xi<+\infty
\end{equation}
(or equivalently
\begin{equation}\label{FBeqn:admvectequivalent2}
C_{\psi,\phi} = \int_{\FBR^{\times}}\int_{\FBR}\overline{\FBcF \psi(A_a{^t\!N_s}\xi)}\FBcF \phi(A_a{^t\!N_s}\xi)\ \frac{\D s\D a}{|a|^\frac{3}{2}},\qquad \text{for a.e.} \; \xi\in\FBR^2 \setminus \{0\}),
\end{equation}
we have the reconstruction formula
\begin{equation}\label{FBreconstructionformulashearlet2}
f = \frac{1}{C_{\psi,\phi}}  \int_{\FBR^{\times}}\int_{\FBR}\int_{\FBR^2}
\FBcS_{\psi}f(b,s,a) \,     \pi_{b,s,a}\phi\ \frac{\D b\D s\D a}{|a|^3}, \qquad
f\in L^2(\FBR^2),
\end{equation}
where again the equality has to be interpreted in the weak-sense ($L^2$-weak convergence). It is immediate to see that if $\psi=\phi$ we recover \eqref{FBreconstructionformulashearlet}. Reconstruction
formula~\eqref{FBreconstructionformulashearlet2} is a consequence of the orthogonality relations for square integrable representations proved by Duflo and Moore in \cite{FBdumo76}. We refer also
to {\cite[Chapter 9]{FBwong02}}. In Section~4 we show that under suitable choices of the vectors $\psi,\phi$ and of the function $f$ formula
\eqref{FBreconstructionformulashearlet2} holds pointwise, see Proposition~\ref{prop:FBreconstructionformulashearletpointwise}.
%From now on, when we consider an admissible vector $\psi$, we suppose $C_{\psi}=1$.

\subsection{The spaces}\label{FBsec:2}
In this subsection we introduce the functional analytic background for our investigations, and state some auxiliary results
which will be used in the proofs of our main results.

The Lizorkin space $\mathcal{S}_0 (\mathbb{R}^{d})$ consists of rapidly decreasing functions with vanishing moments of any order, i.e.
\begin{equation*}\mathcal{S}_0 (\mathbb{R}^{d}) = \left\{\varphi\in
\mathcal{S}(\mathbb{R}^{d}): \: \mu_m (\varphi) = 0,\ \forall m
\in \mathbb{N}^{d} \right\},
\end{equation*}
where $\mu_{m}(\varphi)=\int_{\mathbb{R}^{d}} x^{m}\varphi(x)dx$,
$m\in\mathbb{N}^{d}$.
It is a closed subspace of
$\mathcal{S}(\mathbb{R}^{d})$ equipped with the relative
topology inhered from $\mathcal{S}(\mathbb{R}^{d})$ and
its dual space of Lizorkin distributions $\FBcS'_0(\FBR^d)$ is canonically isomorphic to the quotient of $\FBcS'(\FBR^d)$ by the space of polynomials
(cf. \cite{FBhol1995, FBkpsv2014}).
\par
We are also interested in the Fourier Lizorkin space $\hat{\mathcal{S}}_0 (\mathbb{R}^{d})$ which consists of rapidly decreasing functions that vanish in zero together with all of their partial derivatives, i.e.
\begin{equation*}
\hat{\mathcal{S}}_0 (\mathbb{R}^{d})
= \left\{\varphi\in \mathcal{S}(\mathbb{R}^{d}): \: \partial ^{m} \varphi (0) = 0,\ \forall m \in \mathbb{N}^{d} \right\}.
\end{equation*}
This is also a closed subspace of $\mathcal{S}(\mathbb{R}^{d})$ endowed with the relative
topology inhered from $\mathcal{S}(\mathbb{R}^{d})$.

\par

We observe that, since $\mathcal{S}_0 (\mathbb{R}^{d})$ and  $\hat{\mathcal{S}}_0 (\mathbb{R}^{d})$ are
closed subspaces of the nuclear space $ \mathcal{S} (\mathbb{R}^d) $, they are nuclear as well.
We denote by $X\hat{\otimes}Y$ the topological tensor product space obtained as the completion of the tensor product
$X\otimes Y$ in the inductive tensor product topology $\varepsilon$ or the projective tensor product topology $\pi$.
Then, we have the following result.
\begin{lemma} \label{FBLm:stability}
The spaces $\mathcal{S}_0 (\mathbb{R}^{d})$ and $\hat{\mathcal{S}}_0 (\mathbb{R}^{d})$ are
closed under translations, dilations, differentiations and multiplications by  a polynomial.
Moreover, the Fourier transform is an isomorphism between $\mathcal{S}_0 (\mathbb{R}^{d})$ and
$\hat{\mathcal{S}}_0 (\mathbb{R}^{d})$ and we have the following canonical isomorphisms
\begin{equation*}
\mathcal{S}_{0} (\mathbb{R}^{d}) \cong
\mathcal{S}_{0} (\mathbb{R}^{d_1}) \hat{\otimes} \mathcal{S}_{0} (\mathbb{R}^{d_2}),
%= \mathcal{S}_{0} (\mathbb{R}^{d_1}) \hat{\otimes}_\varepsilon \mathcal{S}_{0} (\mathbb{R}^{d_2}) =
%\mathcal{S}_{0} (\mathbb{R}^{d_1}) \hat{\otimes}_\pi \mathcal{S}_{0} (\mathbb{R}^{d_2}),
\end{equation*}
\begin{equation*}
\hat{\mathcal{S}}_{0} (\mathbb{R}^{d}) \cong
\hat{\mathcal{S}}_{0} (\mathbb{R}^{d_1}) \hat{\otimes} \hat{\mathcal{S}}_{0} (\mathbb{R}^{d_2}),
%= \hat{\mathcal{S}}_{0} (\mathbb{R}^{d_1}) \hat{\otimes}_\varepsilon \hat{\mathcal{S}}_{0} (\mathbb{R}^{d_2}) =
%\hat{\mathcal{S}}_{0} (\mathbb{R}^{d_1}) \hat{\otimes}_\pi \hat{\mathcal{S}}_{0} (\mathbb{R}^{d_2}),
\end{equation*}
where $d =d_1 + d_2 \in \mathbb{Z}_+$.
\end{lemma}
\begin{proof}
The proof is based on classical arguments and we omit it (cf.  \cite[Theorem 51.6]{FBTreves1967} for the canonical isomorphisms).
\end{proof}
\par
The next Lemma is a reformulation of {\cite[Theorem 6.2]{FBmallat09}}.
\begin{lemma}\label{FBLm:mainlm}
Let $f\in   \mathcal{S}_{0} (\mathbb{R}^d) $. Then, for any $m\in\FBN^d$ there exists $g\in  \mathcal{S}_{0} (\mathbb{R}^d)$ such that
\begin{equation*}
\FBcF f (\xi) = \xi ^m \FBcF g (\xi),\qquad \xi \in \mathbb{R}^d,
\end{equation*}
and vice versa. Furthermore, for such $f$ and $g$ and for every $\nu_1\in\FBN$, there exists $\nu_2\in\FBN$ and a constant $C>0$ such that
\begin{equation}\label{eq:equationveryimportant}
\rho_{\nu_1}(g)\leq C\rho_{\nu_2}(f).
\end{equation}
\end{lemma}

\begin{proof}
We start proving the first part of the statement for $d=1$ and $m=1$. Let $f\in   \mathcal{S}_{0} (\mathbb{R}) $ and consider
\begin{equation*}
g(x)=\int_{-\infty}^xf(t)\D t=-\int_{x}^{+\infty}f(t)\D t.
\end{equation*}
%where in the second equality we  use the fact that $f\in \mathcal{S}_{0} (\mathbb{R})$.
%%%%%%%%%%
%For every $k\in\FBN$ and $x>0$ we have that
%\begin{align*}
%\langle x \rangle^{k} |g(x)|&=|\int_{x}^{+\infty}(1+x^2)^{\frac{k}{2}}f(t)\D t|\leq\int_{x}^{+\infty}(1+t^2)^{\frac{k}{2}}|f(t)|\D t\\
%&\leq\rho_{2k+4}(f)\int_{-\infty}^{+\infty}(1+t^2)^{\frac{k}{2}}\frac{1}{(1+t^2)^{k+2}}\D t<+\infty.
%\end{align*}
%Analogously, for every $k\in\FBN$ and $x<0$ we have that
%\begin{align*}
%\langle x \rangle^{k} |g(x)|&=|\int_{-\infty}^{x}(1+x^2)^{\frac{k}{2}}f(t)\D t|\leq\int_{-\infty}^{x}(1+t^2)^{\frac{k}{2}}|f(t)|\D t\\
%&\leq\rho_{2k+4}(f)\int_{-\infty}^{+\infty}(1+t^2)^{\frac{k}{2}}\frac{1}{(1+t^2)^{k+2}}\D t<+\infty.
%\end{align*}
%%%%%%%%%%%%%%
It is easy to show that $g$ is a well defined function and
\begin{equation}\label{eq:lemmaimportant1}
\sup_{x\in\FBR}\langle x \rangle^{k} |g(x)|\lesssim
\rho_{2k+4}(f)<+\infty,
\end{equation}
for every $k\in\FBN$ (cf. \cite{FBbarpilnen}).
Furthermore, $ g'(x) = f(x),$ and then
\begin{equation}\label{eq:lemmaimportant2}
\sup_{x\in\FBR}\langle x \rangle^{k} |g^{(l) }(x)|=\sup_{x\in\FBR}\langle x \rangle^{k} |f^{(l-1) }(x)|<+\infty,
\end{equation}
for every $k\in\FBN$
and $l\geq 1$.
Moreover, since $f\in\FBcS_0(\FBR)$, for any $n\in\FBN$ we compute
\begin{equation*}
\int_{-\infty}^{+\infty}x^ng(x)\D x=-\int_{-\infty}^{+\infty}x^{n+1}g'(x)\D x=-\int_{-\infty}^{+\infty}x^{n+1}f(x)\D x=0.
\end{equation*}
Then, $g\in\FBcS_0(\FBR)$ and by the definition of $g$ we have
$$
\FBcF f(\xi)=\FBcF g'(\xi)=(2\pi i) \xi \FBcF g(\xi),\qquad\xi\in\FBR.
$$

The opposite direction is obviously true since the space $\FBcS_0(\FBR)$ is closed under multiplication by a polynomial, and this concludes the proof of the first part of the statement for $d=1$ and $m=1$. The analogous statement holds true for $|m|>1$ by iterating the above proof $|m|$-times. The case $d>1$ follows by analogous but cumbersome computations, and it is therefore omitted.

\par

Finally, \eqref{eq:equationveryimportant} follows immediately by \eqref{eq:lemmaimportant1} and \eqref{eq:lemmaimportant2}. More precisely, if $f\in   \mathcal{S}_{0} (\mathbb{R}) $ and
\[
\FBcF f (\xi) = \xi ^m \FBcF g (\xi), \;\;\; \xi \in \mathbb{R},
\]
for some  $g\in  \mathcal{S}_{0} (\mathbb{R})$ and  $m\in\FBN$, then
by \eqref{eq:lemmaimportant1} and \eqref{eq:lemmaimportant2} for every $\nu\in\mathbb{N}$ we have
\begin{align*}
\rho_\nu (g) = \sup_{x \in \FBR, l\leq \nu} \langle x \rangle ^\nu |g^{(l)} (x)|\lesssim\rho_{2\nu+4}(f).
\end{align*}
and this concludes the second part of the statement for $d=1$. The analogous result holds true for $d>1$ by using similar arguments.
\end{proof}
\par
By Lemma~\ref{FBLm:mainlm}, if $f\in   \mathcal{S}_{0} (\mathbb{R}^2) $, then for any given $k,l \in \FBN $ there exists $g \in   \mathcal{S}_{0} (\mathbb{R}^2) $ such that
\begin{equation*}
\FBcF f (\xi_1, \xi_2) = \xi_1 ^k \xi_2 ^l \FBcF g (\xi_1, \xi_2), \;\;\; (\xi_1, \xi_2) \in \mathbb{R}^2.
\end{equation*}
Moreover, it is worth observing that by Lemmas \ref{FBLm:stability} and \ref{FBLm:mainlm},   $f\in   \mathcal{S}_{0} (\mathbb{R}^d) $ if and only if it satisfies the directional vanishing moments
\[
\int_{\mathbb{R}} x_j ^{m} f(x_1, x_2, \dots, x_d) dx_j = 0,
\]
for all $m \in \mathbb{N}$ and for every $j\in\{1,\dots,d\}$.
As noted in \cite{FBgr12},
a sufficient number of directional vanishing moments imposes the crucial frequency localization property needed for the detection of anisotropic structures.

\par
%The space $ \mathcal{S} (\mathbb {H}^{(d,d-1,1)}) $ of highly localized
%functions over the half-space (see also \cite{FBhol1995}) consists of  the functions $ \Phi
%\in C^{\infty} (\mathbb{H}^{(d,d-1,1)}) $ such that the seminorms
%\begin{align*}
%&\rho_{k_1,k_2,l,m} ^{\alpha_1,\alpha_2, \beta,\gamma}(\Phi)=\\
%&\sup_{(b,s,a)\in
%\mathbb {H}^{(d,d-1,1)}}\langle b_1 \rangle ^{k_1} \,\langle \tilde {b} \rangle ^{k_2} \,\langle s \rangle ^l \,
%\left(a^{m}+\frac
%{1}{a^{m}}\right)\left| \partial^{\gamma} _{a}  \partial^{\beta} _{s}
% \partial^{\alpha_2} _{\tilde b}
%\partial^{\alpha_1} _{b_1}
%\Phi (b,s,a)\right |
%\end{align*}
%are finite for all $k_1,m,\alpha_1,\gamma \in \mathbb{N}$, $k_2,l,\alpha_2, \beta \in \mathbb{N}^{d-1}$ and where $b=(b_1,\tilde b)\in\FBR\times\FBR^{d-1}$.
%In particular, when $d=2$, we denote $\mathbb{S}:=\mathbb{H}^{(d,d-1,1)} = \FBR^2\times\FBR\times\FBR_+$ and $ \mathcal{S} (\mathbb{S}) $ consists of the

In order to study qualitative properties of the range of the shearlet transform, we introduce appropriate norms which measure
regularity and decay properties with respect to all of the parameters. To this aim it is necessary to involve eight indices. More precisely,
we denote by $ \mathcal{S} (\mathbb{S}) $ the space of functions
$ \Phi\in C^{\infty} (\mathbb{S}) $ such that the norms
\begin{align} \label{FBnorma-hl2}
\nonumber&p_{k_1,k_2,l,m}^{\alpha_1,\alpha_2, \beta,\gamma}(\Phi) \\
&=\sup_{((b_1,b_2),s,a)\in\mathbb{S}}
\langle b_1 \rangle ^{k_1} \,\langle b_2 \rangle ^{k_2} \,\langle s \rangle ^l \, \left(|a|^{m}+\frac {1}{|a|^{m}}\right)
\left|  \partial^{\gamma} _{a}  \partial^{\beta} _{s} \partial^{\alpha_2} _{b_2} \partial^{\alpha_1} _{b_1} \Phi ((b_1,b_2),s,a)\right |
\end{align}
are finite for all $k_1,k_2,l,m,\alpha_1,\alpha_2,\beta,\gamma \in \mathbb{N}$.
The topology of  $ \mathcal{S} (\mathbb{S}) $  is defined by means of the norms \eqref{FBnorma-hl2}. Its dual $ \mathcal{S}' (\mathbb{S}) $ will play a crucial role in
the definition of the shearlet transform of Lizorkin distributions
since it contains the range of this transform.
%We fix ${\rm d}\mu_{\mathbb{S}}(b,s,a)=a^{-3}\D b\D s\D a$ as the standard measure on $\mathbb{\mathbb{S}}$.

\par

If $F$ is a function of at most polynomial growth on $\mathbb{S}$, i.e., if there exist
positive constants $\nu_1,\nu_2,\nu_3$ and $C$ such that
\begin{equation*}
|F(b,s,a)|\leq C \langle b \rangle^{\nu_1} \langle s \rangle^{\nu_2} \left(|a|^{\nu_3}+\frac
{1}{|a|^{\nu_3}}\right),\qquad (b,s,a)\in \mathbb{S},
\end{equation*}
then we identify $F$ with an element of $ \mathcal{S}' (\mathbb{\mathbb{S}}) $ as follows
\begin{equation}\label{FBdualitysynthesisspace}
_{\FBcS'(\mathbb{S})}( F,\Phi )_{\FBcS(\mathbb{S})}= \int_{\FBR^\times}\int_{\FBR}\int_{\FBR^2}F(b,s,a)\Phi(b,s,a)\frac{\D b\D s\D a}{|a|^3},
\end{equation}
for every $ \Phi\in  \mathcal{S} (\mathbb{\mathbb{S}})$.

\section{The continuity of the Shearlet Transform}

We are now ready to state our first main result.
Let $f, \psi\in L^2(\FBR^2)$. By Definition \ref{FBdefsheartransf}, the Plancherel theorem, and formula~\eqref{FBshearfreq} we obtain
\begin{align}\label{eq:shearlettransformfrequency}
\nonumber\mathcal S_\psi f(b,s,a)&=\langle \FBcF f,\FBcF \pi_{b,s,a},\psi\rangle\\
&=|a|^{\frac{3}{4}}\int_{\FBR}\int_{\FBR}\mathcal{F} f(\xi_1,\xi_2)e^{2\pi i (b_1\xi_1+b_2\xi_2)}\overline{\mathcal{F}\psi(a\xi_1,a|a|^{-\frac{1}{2}}(\xi_2-s\xi_1))}{\rm d}\xi_1{\rm d}\xi_2,
\end{align}
for every $(b, s,a) \in\mathbb{S}$.
Moreover, we can consider this transform as a sesquilinear mapping $\mathcal S\colon (f, \psi) \mapsto \mathcal S_\psi f$.
\begin{theorem}\label{thm:continuityshearlettransform}
The operator $ \mathcal S\colon (f, \psi) \mapsto \mathcal S_\psi f$ is a continuous sesquilinear mapping from
$ \mathcal{S}_{0} (\mathbb{R}^2) \times  \FBcS_0 (\FBR^2) $ into $ \mathcal{S}(\mathbb {S}) $. In particular, for every $\psi\in  \FBcS_0 (\FBR^2)$ the shearlet transform
$\mathcal S_\psi\colon L^2(\FBR^2)\to C(\mathbb{S})\cap L^{\infty}(\mathbb{S})$ restricts to
a continuous operator from $ \mathcal{S}_{0} (\mathbb{R}^2)$ into $ \mathcal{S}(\mathbb {S}) $.
\end{theorem}
%Finally, for simplicity, we restrict ourselves to the connected version of the shearlet group, which corresponds to restricting the scale parameter $a$ over $\FBR_+$.

\begin{proof}
%Recall, $\rho_\nu (\varphi ) = \sup_{x \in \FBR^d, |m|\leq \nu} \langle x \rangle ^\nu |\partial ^{m} \varphi (x) |, $ $ \nu \in \FBN ,$ $\varphi \in  \mathcal{S} (\FBR)$.
We must prove that for every $f, \psi\in \mathcal{S}_{0} (\mathbb{R}^2)$, given $k_1,k_2,l,m,\alpha_1,\alpha_2,\beta,\gamma \in \mathbb{N}$, there exist $ \nu_1,\nu_2 \in \FBN$ such that
\begin{equation} \label{eq:normcontinuity}
\rho_{k_1,k_2,l,m} ^{\alpha_1,\alpha_2, \beta,\gamma} ( \mathcal S_\psi f) \lesssim \rho_{\nu_1} (f )\rho_{\nu_2} (\psi).
\end{equation}

First, in three steps, we show that without lost on generality we can assume $\alpha_1=\alpha_2=\beta=\gamma=k_1=k_2=l=0$. Then,
in the last step we finish the proof by  taking an arbitrary $m\in{\FBN}$ and proving
that
\begin{equation*}
\rho_{0,0,0,m} ^{0,0,0,0} ( \mathcal S_\psi f) \lesssim C \rho_{\nu_1} (f )\rho_{\nu_2} (\psi),
\end{equation*}
for some $ \nu_1,\nu_2 \in \FBN$.
The idea of the proof is similar to the one given in \cite{FBkpsv2014} which concerns
the continuity of the ridgelet transform on $\FBcS_0(\FBR^2)$.
We note that formula~\eqref{eq:shearlettransformfrequency} and Lemma~\ref{FBLm:mainlm} play a crucial role in  the proof.

\par

\textbf{1.} We first show that the derivatives of
$ \mathcal S_\psi f $ are bounded by appropriate norms of the family \eqref{FBnorma-hl2}
 with $\alpha_1=\alpha_2=\beta=\gamma=0$.
By formula \eqref{eq:shearlettransformfrequency} we have that
%\begin{align}
%\nonumber&|\partial_{b_1}^{\alpha_1}\mathcal S_\psi f((b_1,b_2),s,a)|\\
%\nonumber&=|a|^{\frac{3}{4}}|\int_{\FBR}\int_{\FBR}\FBcF f(\xi_1,\xi_2)(\partial_{b_1}^{\alpha_1}e^{2\pi i\xi_1 b_1})e^{2\pi i\xi_2b_2}\overline{\FBcF\psi(a\xi_1,|a|^{\frac{1}{2}}(\xi_2-s\xi_1))}\D\xi_1\D\xi_2|\\
%&=|a|^{\frac{3}{4}}|(2\pi i)^{\alpha_1}\int_{\FBR}\int_{\FBR}\xi_1^{\alpha_1}\FBcF f(\xi_1,\xi_2)e^{2\pi i (b_1\xi_1+b_2\xi_2)}\overline{\FBcF\psi(a\xi_1,|a|^{\frac{1}{2}}(\xi_2-s\xi_1))}\D\xi_1\D\xi_2|\label{derivationb1}.
%\end{align}
%Using formula \eqref{derivationb1}, we have
%\begin{align}
%\nonumber&|\partial_{b_2}^{\alpha_2}\partial_{b_1}^{\alpha_1}\mathcal S_\psi f((b_1,b_2),s,a)|\\
%&=|a|^{\frac{3}{4}}|(2\pi i)^{\alpha_1+\alpha_2}\int_{\FBR}\int_{\FBR}\xi_1^{\alpha_1}\xi_2^{\alpha_2}\FBcF f(\xi_1,\xi_2)e^{2\pi i (b_1\xi_1+b_2\xi_2)}\overline{\FBcF\psi(a\xi_1,|a|^{\frac{1}{2}}(\xi_2-s\xi_1))}\D\xi_1\D\xi_2|.\label{derivationb2}
%\end{align}
%Using formula \eqref{derivationb2}, we have
\begin{align}\label{derivations}
\nonumber&|\partial_{s}^{\beta}\partial_{b_2}^{\alpha_2}\partial_{b_1}^{\alpha_1}\mathcal S_\psi f((b_1,b_2),s,a)|\\
\nonumber&=|a|^{\frac{3}{4}}|(2\pi i)^{\alpha_1+\alpha_2}\int_{\FBR}\int_{\FBR}\xi_1^{\alpha_1}\xi_2^{\alpha_2}\FBcF f(\xi_1,\xi_2)
e^{2\pi i (b_1\xi_1+b_2\xi_2)}\times\\
\nonumber&\times\overline{\partial_{s}^{\beta}(\FBcF\psi(a\xi_1,a|a|^{-\frac{1}{2}}(\xi_2-s\xi_1)))}\D\xi_1\D\xi_2|\\
\nonumber&=|a|^{\frac{3}{4}+\frac{\beta}{2}}|(2\pi i)^{\alpha_1+\alpha_2}\int_{\FBR}\int_{\FBR}\xi_1^{\alpha_1+\beta}\xi_2^{\alpha_2}\FBcF f(\xi_1,\xi_2)e^{2\pi i (b_1\xi_1+b_2\xi_2)}\times\\
&\times\overline{\partial_{2}^{\beta}\FBcF\psi(a\xi_1,a|a|^{-\frac{1}{2}}(\xi_2-s\xi_1))}\D\xi_1\D\xi_2|.
\end{align}
Then, by formula \eqref{derivations} and the Liebniz formula, we compute
\begin{align}\label{eq:derivationa}
\nonumber&|\partial_{a}^{\gamma}\partial_{s}^{\beta}\partial_{b_2}^{\alpha_2}\partial_{b_1}^{\alpha_1}\mathcal S_\psi f((b_1,b_2),s,a)|\\
\nonumber&=|(2\pi i)^{\alpha_1+\alpha_2}\int_{\FBR}\int_{\FBR}\xi_1^{\alpha_1+\beta}\xi_2^{\alpha_2}\FBcF f(\xi_1,\xi_2)e^{2\pi i (b_1\xi_1+b_2\xi_2)}\times\\
\nonumber&\times\partial_a^{\gamma}(|a|^{\frac{3}{4}+\frac{\beta}{2}}\overline{\partial_{2}^{\beta}\FBcF\psi(a\xi_1,a|a|^{-\frac{1}{2}}(\xi_2-s\xi_1))})\D\xi_1\D\xi_2|\\
\nonumber&=|\sum_{j+l=\gamma}c_{j,l}(2\pi i)^{\alpha_1+\alpha_2}|a|^{\frac{3}{4}+\frac{\beta}{2}-j}\int_{\FBR} \int_{\FBR} \xi_1^{\alpha_1+\beta}\xi_2^{\alpha_2}\FBcF f(\xi_1,\xi_2)e^{2\pi i (b_1\xi_1+b_2\xi_2)}\times\\
&\times\partial_a^{l}(\overline{\partial_{2}^{\beta}\FBcF\psi(a\xi_1,a|a|^{-\frac{1}{2}}(\xi_2-s\xi_1))})\D\xi_1\D\xi_2|,
\end{align}
where the constants $c_{j,l}$  depend only on $j,\, l$ and $\beta$. The last expression in ~\eqref{eq:derivationa}
is bounded from above by a finite sum of terms of the form
\begin{align*}
&|a|^{\frac{5}{4}+\frac{\beta}{2}-j-3\frac{r_2}{2}-r_1}|c_{j,l}(2\pi i)^{\alpha_1+\alpha_2}\int_{\FBR} \int_{\FBR} \xi_1^{\alpha_1+\beta}\xi_2^{\alpha_2}\FBcF f(\xi_1,\xi_2)e^{2\pi i (b_1\xi_1+b_2\xi_2)}\times\\
&\times a^{r_1+r_2}|a|^{-\frac{r_2}{2}}\overline{\xi_1^{r_1}(\xi_2-s\xi_1)^{r_2}(\partial_1^{r_1}\partial_2^{\beta+r_2}\FBcF\psi(a\xi_1,a|a|^{-\frac{1}{2}}(\xi_2-s\xi_1))})\D\xi_1\D\xi_2|\\
&\lesssim \rho_{0,0,0,|\frac{1}{2}+\frac{\beta}{2}-j-3\frac{r_2}{2}-r_1|}^{0,0,0,0}(\mathcal{S}_{\psi_{r_1,r_2}}(\partial_1^{\alpha_1+\beta}\partial_2^{\alpha_2}f)),
\end{align*}
with $\psi_{r_1,r_2}$ given by $\mathcal{F}\psi_{r_1,r_2}(\xi_1,\xi_2)=\xi_1^{r_1}\xi_2^{r_2}\partial_1^{r_1}\partial_2^{\beta+r_2}\FBcF\psi(\xi_1,\xi_2)$ for every $r_1,\, r_2\in\mathbb{N}$, $r_1+r_2\leq l$.
Since differentiation and multiplication by polynomials are continuous operators from $\mathcal{S}_{0} (\mathbb{R}^2)$ into itself, we can assume $\alpha_1=\alpha_2=\beta=\gamma=0$.

\par

\textbf{2.} In the second step we estimate multiplications by $b = (b_1,b_2) \in \mathbb{R}^2$, and show
that it is enough to consider $k_1=k_2=0$ when proving  \eqref{eq:normcontinuity}.
Observe that we can assume $|b_1| \geq 1,\, |b_2| \geq 1$. By formula \eqref{eq:shearlettransformfrequency}, and since
\[
e^{2\pi i\xi_1 b_1}=\frac{(1-\partial_{\xi_1}^2 )^{N} e^{2\pi i\xi_1 b_1} }{ \langle 2\pi b_1 \rangle ^{2N}}
\]
for every $N\in\mathbb{N}$ and $\xi_1,\, b_1\in\FBR$, it follows that
\begin{align*}
& \langle b_1 \rangle ^{k_1}  |\mathcal S_\psi f((b_1,b_2),s,a)|\\
&=\langle b_1 \rangle ^{k_1}|a|^{\frac{3}{4}}|\int_{\FBR} \int_{\FBR}\mathcal{F} f(\xi_1,\xi_2)\frac{(1-\partial_{\xi_1}^2 )^{N} e^{2\pi i\xi_1 b_1} }{ \langle 2\pi b_1 \rangle ^{2N}}e^{2\pi i b_2\xi_2} \overline{\mathcal{F}\psi(a\xi_1,a|a|^{-\frac{1}{2}}(\xi_2-s\xi_1))}{\rm d}\xi_1{\rm d}\xi_2|\\
&=\frac{\langle b_1 \rangle ^{k_1}}{\langle 2\pi b_1 \rangle ^{2N}}|a|^{\frac{3}{4}}|\int_{\FBR} \int_{\FBR}(1-\partial_{\xi_1}^2 )^{N} (\mathcal{F} f(\xi_1,\xi_2)\overline{\mathcal{F}\psi(a\xi_1,a|a|^{-\frac{1}{2}}(\xi_2-s\xi_1))})e^{2\pi i b\xi} {\rm d}\xi_1{\rm d}\xi_2|\\
&=\frac{\langle b_1 \rangle ^{k_1}}{\langle 2\pi b_1 \rangle ^{2N}}|\sum_{j,r_1,r_2\leq2N}|a|^{\frac{3}{4}+r_1+\frac{r_2}{2}}\int_{\FBR} \int_{\FBR}P_{j,r_1,r_2}(s)\partial_1^{j}\mathcal{F} f(\xi_1,\xi_2)\times\\
&\times\partial_{1}^{r_1}\partial_{2}^{r_2}\overline{\mathcal{F}\psi(a\xi_1,a|a|^{-\frac{1}{2}}(\xi_2-s\xi_1))}e^{2\pi i b\xi}
{\rm d}\xi_1{\rm d}\xi_2|,
\end{align*}
for some polynomials $P_{j,r_1,r_2}(t)=\sum_{m=0}^{p_{j,r_1,r_2}}c_m t^m$, and
$b\xi = b_1\xi_1+b_2\xi_2$. Then, we obtain
\begin{align*}
& \langle b_1 \rangle ^{k_1}  |\mathcal S_\psi f((b_1,b_2),s,a)|\\
& \leq\frac{\langle b_1 \rangle ^{k_1}|a|^{N}\langle s\rangle^{2N}}{\langle 2\pi b_1 \rangle ^{2N}}|\sum_{j,r_1,r_2\leq2N}\sum_{m=0}^{p_{j,r_1,r_2}}c_m|a|^{\frac{3}{4}}\int_{\FBR} \int_{\FBR}\partial_1^{j}\mathcal{F} f(\xi_1,\xi_2)
\times \\
& \times\partial_{1}^{r_1}\partial_{2}^{r_2}\overline{\mathcal{F}\psi(a\xi_1,a|a|^{-\frac{1}{2}}(\xi_2-s\xi_1))}
e^{2\pi i b\xi} {\rm d}\xi_1{\rm d}\xi_2|,
\end{align*}
where the last inequality follows by the fact that we can assume $|a|\geq1$ and $|s|\geq1$.
The above estimate yields
\begin{align*}
& \langle b_1 \rangle ^{k_1}  |\mathcal S_\psi f((b_1,b_2),s,a)|
\leq\frac{\langle b_1 \rangle ^{k_1}}{\langle 2\pi b_1 \rangle ^{2N}}\sum_{j,r_1,r_2\leq2N}\sum_{m=0}^{p_{j,r_1,r_2}}|c_m|
\rho_{0,0,2N,N}^{0,0,0,0}(\mathcal{S}_{\psi_{r_1,r_2}}f_j),
\end{align*}
with $\mathcal{F}f_j(\xi_1,\xi_2)=\partial_1^{j}\mathcal{F} f(\xi_1,\xi_2)$ for every $j$ and
$\psi_{r_1,r_2}$ given by
$\mathcal{F}\psi_{r_1,r_2}(\xi_1,\xi_2)=\partial_{1}^{r_1}\partial_{2}^{r_2}\mathcal{F}\psi(\xi_1,\xi_2)$ for every pair $r_1,r_2 \in
\mathbb{N}$, $r_1,r_2 \leq 2N $.
Analogously, for every $M \in\FBN$ we have that
\begin{align*}
& \langle b_2 \rangle ^{k_2}  |\mathcal S_\psi f((b_1,b_2),s,a)|\\
%&=\langle b_2 \rangle ^{k_2}|a|^{\frac{3}{4}}|\int_{\FBR} \int_{\FBR}\mathcal{F} f(\xi_1,\xi_2)e^{2\pi i (b_1\xi_1+b_2\xi_2)} \overline{\mathcal{F}\psi(a\xi_1,|a|^{\frac{1}{2}}(\xi_2-s\xi_1))}{\rm d}\xi_1{\rm d}\xi_2|\\
&=\langle b_2 \rangle ^{k_2}|a|^{\frac{3}{4}}|\int_{\FBR} \int_{\FBR}\mathcal{F} f(\xi_1,\xi_2)\frac{(1-\partial_{\xi_2}^2 )^{M} e^{2\pi i\xi_2 b_2} }{ \langle 2\pi b_2 \rangle ^{2M}}e^{2\pi i b_1\xi_1} \overline{\mathcal{F}\psi(a\xi_1,a|a|^{-\frac{1}{2}}(\xi_2-s\xi_1))}{\rm d}\xi_1{\rm d}\xi_2|\\
&=\frac{\langle b_2 \rangle ^{k_2}}{\langle 2\pi b_2 \rangle ^{2M}}|a|^{\frac{3}{4}}|\int_{\FBR} \int_{\FBR}(1-\partial_{\xi_2}^2 )^{M} (\mathcal{F} f(\xi_1,\xi_2)\overline{\mathcal{F}\psi(a\xi_1,a|a|^{-\frac{1}{2}}(\xi_2-s\xi_1))})e^{2\pi i b\xi} {\rm d}\xi_1{\rm d}\xi_2|\\
&\lesssim\frac{\langle b_2 \rangle ^{k_2}}{\langle 2\pi b_2 \rangle ^{2M}}|\sum_{j,r\leq2M}|a|^{\frac{3}{4}+\frac{r}{2}}\int_{\FBR} \int_{\FBR}\partial_2^{j}\mathcal{F} f(\xi_1,\xi_2)\partial_{2}^{r}\overline{\mathcal{F}\psi(a\xi_1,a|a|^{-\frac{1}{2}}(\xi_2-s\xi_1))}e^{2\pi i b\xi} {\rm d}\xi_1{\rm d}\xi_2|\\
&\lesssim\frac{\langle b_2 \rangle ^{k_2}|a|^M}{\langle 2\pi b_2 \rangle ^{2M}}|\sum_{j,r\leq2M}|a|^{\frac{3}{4}}\int_{\FBR} \int_{\FBR}\partial_2^{j}\mathcal{F} f(\xi_1,\xi_2)\partial_{2}^{r}\overline{\mathcal{F}\psi(a\xi_1,a|a|^{-\frac{1}{2}}(\xi_2-s\xi_1))}e^{2\pi i b\xi} {\rm d}\xi_1{\rm d}\xi_2|
\end{align*}
where the last inequality follows by assuming $|a|\geq1$. Hence, we obtain
\begin{align*}
& \langle b_2 \rangle ^{k_2}  |\mathcal S_\psi f((b_1,b_2),s,a)|
\leq\frac{\langle b_2 \rangle ^{k_2}}{\langle 2\pi b_2 \rangle ^{2M}}\sum_{j,r\leq2M}
\rho_{0,0,0,M}^{0,0,0,0}(\mathcal{S}_{ \psi_r} \tilde f_j),
\end{align*}
with $ \tilde f_j$ given by $\mathcal{F} \tilde f_j(\xi_1,\xi_2)=\partial_2^{j}\mathcal{F} f(\xi_1,\xi_2)$,
and $\psi_r$ given by
$\mathcal{F} \psi_r (\xi_1,\xi_2)=\partial_{2}^{r}\mathcal{F}\psi(\xi_1,\xi_2)$ for every $j, r \in \FBN$, $ j,r \leq 2M$.

Therefore, recalling that we have assumed $|b_1|\geq1$ and $|b_2|\geq1$, if we choose $N=k_1$ and $M=k_2$, we obtain
\begin{align*}
&\langle b_1 \rangle ^{k_1} \langle b_2 \rangle ^{k_2}  |\mathcal S_\psi f((b_1,b_2),s,a)|^2
\lesssim\frac{\langle b_1 \rangle ^{k_1} \langle b_2 \rangle ^{k_2}}{\langle b_1 \rangle ^{2k_1} \langle b_2 \rangle ^{2k_2}}\times\\
&\times [\sum_{j,k_1,k_2\leq2N}\sum_{m=0}^{p_{j,k_1,k_2}}|c_m|
\rho_{0,0,2N,N}^{0,0,0,0}(\mathcal{S}_{\psi_{r_1,r_2}}f_j)]
[ \sum_{j,r\leq2M}
\rho_{0,0,0,M}^{0,0,0,0}(\mathcal{S}_{ \psi_r } \tilde f_j)]\\
&\lesssim  [\sum_{j,k_1,k_2\leq2N}\sum_{m=0}^{p_{j,k_1,k_2}}|c_m|
\rho_{0,0,2N,N}^{0,0,0,0}(\mathcal{S}_{\psi_{r_1,r_2}}f_j)]
[ \sum_{j,r\leq2M}
\rho_{0,0,0,M}^{0,0,0,0}(\mathcal{S}_{ \psi_r } \tilde f_j)].
\end{align*}
and we conclude that we can assume $k_1=k_2=0$ 	since differentiation and multiplication by polynomials are continuous operators from $\mathcal{S}_{0} (\mathbb{R}^2)$ into itself.
\par
\textbf{3.} The third step consists in showing that when considering multiplications by powers of $s$, we can assume $l=0$ in \eqref{eq:normcontinuity}.
By formula \eqref{eq:shearlettransformfrequency} and
\[
\langle s \rangle \leq \langle s\xi_1 \rangle \left(\frac{1+\xi_1^2}{\xi_1^2}\right),\qquad s,\, \xi_1\in\FBR,
\]
it follows that
\begin{align*}
& \langle s \rangle ^{l}  |\mathcal S_\psi f((b_1,b_2),s,a)|\\
&=|a|^{\frac{3}{4}}|\int_{\FBR} \int_{\FBR} \mathcal{F} f(\xi_1,\xi_2)e^{2\pi i (b_1\xi_1+b_2\xi_2)}\langle s \rangle ^{l} \overline{\mathcal{F}\psi(a\xi_1,a|a|^{-\frac{1}{2}}(\xi_2-s\xi_1))}{\rm d}\xi_1{\rm d}\xi_2|\\
&\leq|a|^{\frac{3}{4}}|\int_{\FBR} \int_{\FBR} \xi_1^{-2l}\langle\xi_1\rangle^{2l}\mathcal{F} f(\xi_1,\xi_2)e^{2\pi i (b_1\xi_1+b_2\xi_2)}\langle s\xi_1 \rangle ^{l} \overline{\mathcal{F}\psi(a\xi_1,a|a|^{-\frac{1}{2}}(\xi_2-s\xi_1))}{\rm d}\xi_1{\rm d}\xi_2|.
\end{align*}
By Lemma \ref{FBLm:mainlm} there exists $f_{l}\in   \mathcal{S}_0 (\mathbb{R}^2)  $ such that
\[
\FBcF f(\xi_1,\xi_2) =\xi_1^{2l} \FBcF f_{l}(\xi_1,\xi_2),
\]
for every $(\xi_1,\xi_2)\in\FBR^2$, and then
\begin{align*}
& \langle s \rangle ^{l}  |\mathcal S_\psi f((b_1,b_2),s,a)|\\
&\leq|a|^{\frac{3}{4}}|\int_{\FBR} \int_{\FBR} \langle\xi_1\rangle^{2l}\mathcal{F} f_l(\xi_1,\xi_2)e^{2\pi i (b_1\xi_1+b_2\xi_2)}\langle s\xi_1 \rangle ^{l} \overline{\mathcal{F}\psi(a\xi_1,a|a|^{-\frac{1}{2}}(\xi_2-s\xi_1))}{\rm d}\xi_1{\rm d}\xi_2|.
\end{align*}

Without loss of generality, we may assume that $l$ is even. We divide the proof in the two cases $|a|\leq1$ and $|a|>1$.
 If $|a|\geq1$, by Peetre's inequality we have that
\begin{equation*}
 \langle s\xi_1 \rangle ^{l}\leq  \langle a|a|^{-\frac{1}{2}}s\xi_1 \rangle ^{l}\leq2^l\langle a|a|^{-\frac{1}{2}}(\xi_2-s\xi_1)\rangle^l \langle a|a|^{-\frac{1}{2}}\xi_2 \rangle ^{l},\qquad s,\, \xi_1,\, \xi_2\in\FBR.
\end{equation*}
Then
\begin{align*}
& \langle s \rangle ^{l}  |\mathcal S_\psi f((b_1,b_2),s,a)|\\
%&\leq |a|^{\frac{3}{4}}|\int_{\FBR} \int_{\FBR} \mathcal{F} f(\xi_1,\xi_2)e^{2\pi i (b_1\xi_1+b_2\xi_2)}\langle s \rangle ^{l} \overline{\mathcal{F}\psi(a\xi_1,|a|^{\frac{1}{2}}(\xi_2-s\xi_1))}{\rm d}\xi_1{\rm d}\xi_2|\\
&\leq|a|^{\frac{3}{4}}|\int_{\FBR} \int_{\FBR} 2^l\langle\xi_1\rangle^{2l}\langle a|a|^{-\frac{1}{2}}\xi_2 \rangle ^{l}\mathcal{F} f_l(\xi_1,\xi_2)e^{2\pi i (b_1\xi_1+b_2\xi_2)}\times\\
&\times\langle a|a|^{-\frac{1}{2}}(\xi_2-s\xi_1)\rangle^l \overline{\mathcal{F}\psi(a\xi_1,a|a|^{-\frac{1}{2}}(\xi_2-s\xi_1))}{\rm d}\xi_1{\rm d}\xi_2|,
%&\lesssim\sum_{m=0}^{p_l}c_m|a|^{\frac{3}{4}+\frac{m}{2}}|\int_{\FBR} \int_{|\xi_1|\geq|a|^{-\frac{1}{2}}} \xi_2^m \mathcal{F} f(\xi_1,\xi_2)e^{2\pi i (b_1\xi_1+b_2\xi_2)}\times\\
%&\times\langle|a|^\frac{1}{2}(\xi_2-s\xi_1)\rangle^l \overline{\mathcal{F}\psi(a\xi_1,|a|^{\frac{1}{2}}(\xi_2-s\xi_1))}{\rm d}\xi_1{\rm d}\xi_2|\\
\end{align*}
which is less than or equal to a finite sum of terms of the form
\begin{align*}
%& \langle s \rangle ^{l}  |\mathcal S_\psi f((b_1,b_2),s,a)|\\
%&\leq |a|^{\frac{3}{4}}|\int_{\FBR} \int_{\FBR} \mathcal{F} f(\xi_1,\xi_2)e^{2\pi i (b_1\xi_1+b_2\xi_2)}\langle s \rangle ^{l} \overline{\mathcal{F}\psi(a\xi_1,|a|^{\frac{1}{2}}(\xi_2-s\xi_1))}{\rm d}\xi_1{\rm d}\xi_2|\\
&|a|^{\frac{3}{4}+\frac{r_2}{2}}|\int_{\FBR} \int_{\FBR} 2^l\xi_1^{r_1}\xi_2^{r_2}\mathcal{F} f_l(\xi_1,\xi_2)e^{2\pi i (b_1\xi_1+b_2\xi_2)}\times\\
&\times\langle a|a|^{-\frac{1}{2}}(\xi_2-s\xi_1)\rangle^l \overline{\mathcal{F}\psi(a\xi_1,a|a|^{-\frac{1}{2}}(\xi_2-s\xi_1))}{\rm d}\xi_1{\rm d}\xi_2|
%&\lesssim\sum_{m=0}^{p_l}c_m|a|^{\frac{3}{4}+\frac{m}{2}}|\int_{\FBR} \int_{|\xi_1|\geq|a|^{-\frac{1}{2}}} \xi_2^m \mathcal{F} f(\xi_1,\xi_2)e^{2\pi i (b_1\xi_1+b_2\xi_2)}\times\\
%&\times\langle|a|^\frac{1}{2}(\xi_2-s\xi_1)\rangle^l \overline{\mathcal{F}\psi(a\xi_1,|a|^{\frac{1}{2}}(\xi_2-s\xi_1))}{\rm d}\xi_1{\rm d}\xi_2|\\
\lesssim \rho_{0,0,0,\frac{r_2}{2}}^{0,0,0,0}(\mathcal{S}_{\tilde{\psi}}(\partial_1^{r_1}\partial_2^{r_2}f_l)),
\end{align*}
where $ 0 \leq  r_1,\, r_2 \leq 3l$ and with $\tilde \psi$ given by $\mathcal{F}\tilde{\psi}(\xi_1,\xi_2)=\langle\xi_2\rangle^{l}\FBcF\psi(\xi_1,\xi_2)$.
%Therefore, by Lemma \ref{FBLm:mainlm} and since differentiation and multiplication by polynomials are continuous operators from $\mathcal{S}_{0} (\mathbb{R}^2)$ into itself, we can assume $\alpha_1=\alpha_2=\beta=\gamma=0$.
Analogously, if $|a|\leq1$, by Peetre's inequality we have
\begin{equation*}
\langle s\xi_1 \rangle ^{l}\leq|a|^{-\frac{l}{2}} \langle a|a|^{-\frac{1}{2}}s\xi_1 \rangle ^{l}\leq |a|^{-\frac{l}{2}} 2^l\langle a|a|^{-\frac{1}{2}}(\xi_2-s\xi_1)\rangle^l \langle a|a|^{-\frac{1}{2}}\xi_2 \rangle ^{l},\qquad s,\, \xi_1,\, \xi_2\in\FBR,
\end{equation*}
and then
\begin{align*}
& \langle s \rangle ^{l}  |\mathcal S_\psi f((b_1,b_2),s,a)|\\
%&\leq |a|^{\frac{3}{4}}|\int_{\FBR} \int_{\FBR} \mathcal{F} f(\xi_1,\xi_2)e^{2\pi i (b_1\xi_1+b_2\xi_2)}\langle s \rangle ^{l} \overline{\mathcal{F}\psi(a\xi_1,|a|^{\frac{1}{2}}(\xi_2-s\xi_1))}{\rm d}\xi_1{\rm d}\xi_2|\\
&\leq|a|^{\frac{3}{4}-\frac{l}{2}}|\int_{\FBR} \int_{\FBR} 2^l\langle\xi_1\rangle^{2l}\langle a|a|^{-\frac{1}{2}}\xi_2 \rangle ^{l}\mathcal{F} f_l(\xi_1,\xi_2)e^{2\pi i (b_1\xi_1+b_2\xi_2)}\times\\
&\times\langle a|a|^{-\frac{1}{2}}(\xi_2-s\xi_1)\rangle^l \overline{\mathcal{F}\psi(a\xi_1,a|a|^{-\frac{1}{2}}(\xi_2-s\xi_1))}{\rm d}\xi_1{\rm d}\xi_2|,
%&\lesssim\sum_{m=0}^{p_l}c_m|a|^{\frac{3}{4}+\frac{m}{2}}|\int_{\FBR} \int_{|\xi_1|\geq|a|^{-\frac{1}{2}}} \xi_2^m \mathcal{F} f(\xi_1,\xi_2)e^{2\pi i (b_1\xi_1+b_2\xi_2)}\times\\
%&\times\langle|a|^\frac{1}{2}(\xi_2-s\xi_1)\rangle^l \overline{\mathcal{F}\psi(a\xi_1,|a|^{\frac{1}{2}}(\xi_2-s\xi_1))}{\rm d}\xi_1{\rm d}\xi_2|\\
\end{align*}
which is less than or equal to a finite sum of terms of the form
\begin{align*}
%& \langle s \rangle ^{l}  |\mathcal S_\psi f((b_1,b_2),s,a)|\\
%&\leq |a|^{\frac{3}{4}}|\int_{\FBR} \int_{\FBR} \mathcal{F} f(\xi_1,\xi_2)e^{2\pi i (b_1\xi_1+b_2\xi_2)}\langle s \rangle ^{l} \overline{\mathcal{F}\psi(a\xi_1,|a|^{\frac{1}{2}}(\xi_2-s\xi_1))}{\rm d}\xi_1{\rm d}\xi_2|\\
&|a|^{\frac{3}{4}-\frac{l}{2}+\frac{r_4}{2}}|\int_{\FBR} \int_{\FBR} 2^l\xi_1^{r_3}\xi_2^{r_4}\mathcal{F} f_l(\xi_1,\xi_2)e^{2\pi i (b_1\xi_1+b_2\xi_2)}\times\\
&\times\langle a|a|^{-\frac{1}{2}}(\xi_2-s\xi_1)\rangle^l \overline{\mathcal{F}\psi(a\xi_1,a|a|^{-\frac{1}{2}}(\xi_2-s\xi_1))}{\rm d}\xi_1{\rm d}\xi_2|
%&\lesssim\sum_{m=0}^{p_l}c_m|a|^{\frac{3}{4}+\frac{m}{2}}|\int_{\FBR} \int_{|\xi_1|\geq|a|^{-\frac{1}{2}}} \xi_2^m \mathcal{F} f(\xi_1,\xi_2)e^{2\pi i (b_1\xi_1+b_2\xi_2)}\times\\
%&\times\langle|a|^\frac{1}{2}(\xi_2-s\xi_1)\rangle^l \overline{\mathcal{F}\psi(a\xi_1,|a|^{\frac{1}{2}}(\xi_2-s\xi_1))}{\rm d}\xi_1{\rm d}\xi_2|\\
\lesssim \rho_{0,0,0,|\frac{r_4}{2}-\frac{l}{2}|}^{0,0,0,0}(\mathcal{S}_{\tilde{\psi}}(\partial_1^{r_3}\partial_2^{r_4}f_l)),
\end{align*}
where $ 0 \leq r_3,\, r_4  \leq 3l$.
Therefore, by Lemma \ref{FBLm:mainlm}, and since differentiation and multiplication by polynomials are continuous operators from $\mathcal{S}_{0} (\mathbb{R}^2)$ into itself, we can assume $l=0$.

\par

\textbf{4.} Finally, we consider multiplication by positive and negative powers of $|a|,\, a \in\mathbb{R}^{\times}$. Let $m\in\mathbb{N}$.
%\begin{align*}
%&|a|^m |\mathcal S_\psi f((b_1,b_2),s,a)|\\
%&=|a|^{\frac{3}{4}+m}|\int_{\FBR}\int_{\FBR}\FBcF f(\xi_1,\xi_2)e^{2\pi i (b_1\xi_1+b_2\xi_2)}\overline{\FBcF\psi(a\xi_1,|a|^{\frac{1}{2}}(\xi_2-s\xi_1))}\D\xi_1\D\xi_2|\\
%\end{align*}
By Lemma \ref{FBLm:mainlm} there exists $g_m \in   \mathcal{S}_{0} (\mathbb{R}^2) $ such that
\[
\FBcF f (\xi_1, \xi_2) = \xi_1 ^m \FBcF g_m (\xi_1, \xi_2),
\]
for every $(\xi_1,\xi_2)\in\FBR^2$. Then, assuming $|a|\geq1$, we have
\begin{align*}
&|a|^m |\mathcal S_\psi f((b_1,b_2),s,a)|\\
&=|a|^{\frac{3}{4}}|\int_{\FBR}\int_{\FBR}\FBcF g_m(\xi_1,\xi_2)e^{2\pi i (b_1\xi_1+b_2\xi_2)}a^m\xi_1^m\overline{\FBcF\psi(a\xi_1,a|a|^{-\frac{1}{2}}(\xi_2-s\xi_1))}\D\xi_1\D\xi_2|\\
&\lesssim |a|^{\frac{3}{4}}\int_{\FBR}\int_{\FBR}|\FBcF g_m(\xi_1,\xi_2)||\FBcF(\partial_{1}^m\psi(a\xi_1,a|a|^{-\frac{1}{2}}(\xi_2-s\xi_1)))|\D\xi_1\D\xi_2\\
&\lesssim \|g_m\|_1|a|^{\frac{3}{4}}\int_{\FBR}\int_{\FBR}|\FBcF(\partial_{1}^m\psi(a\xi_1,a|a|^{-\frac{1}{2}}(\xi_2-s\xi_1)))|\D\xi_1\D\xi_2\\
&\lesssim \|g_m\|_1|a|^{-\frac{3}{4}}\int_{\FBR}\int_{\FBR}|\FBcF(\partial_{1}^m\psi(\xi_1,\xi_2))|\D\xi_1\D\xi_2\leq C \|g_m\|_1|a|^{-\frac{3}{4}}\|\FBcF\partial_{1}^m\psi\|_1\\
&\lesssim \rho_2(g_m)\rho_{2}(\FBcF\partial_{1}^m\psi).
\end{align*}

Therefore, by Lemma~\ref{FBLm:mainlm}, and since the Fourier transform and the differentiation are continuous operators from $\mathcal{S}_{0} (\mathbb{R}^2)$ into itself,  we conclude that
there exist $\nu_1,\, \nu_2\in\mathbb{N}$
such that
\begin{align}\label{eq:inequalitypositivepowera}
|a|^m |\mathcal S_\psi f((b_1,b_2),s,a)|\lesssim \rho_{\nu_1}(f)\rho_{\nu_2}(\psi)
\end{align}
for every $(b_1,b_2)\in\FBR^2,\, s\in\FBR$ and $|a|\geq1$.
\par
We finish the proof by considering multiplications by negative powers of $|a|$. Let $m\in\mathbb{N}$.
%\begin{align*}
%&|a^{-m} \mathcal S_\psi f((b_1,b_2),s,a)|=|a^{\frac{3}{4}-m}\int_{\FBR}\int_{\FBR}\FBcF f(\xi_1,\xi_2)e^{2\pi i (b_1\xi_1+b_2\xi_2)}\overline{\FBcF\psi(a\xi_1,a^{\frac{1}{2}}(\xi_2-s\xi_1))}\D\xi_1\D\xi_2|
%\end{align*}
By Lemma \ref{FBLm:mainlm} there exists $\psi_{m}\in   \mathcal{S}_0 (\mathbb{R}^2)  $ such that
\[
\FBcF \psi(\xi_1,\xi_2) =\xi_1^{m+1} \FBcF \psi_{m}(\xi_1,\xi_2),
\]
for every $(\xi_1,\xi_2)\in\FBR^2$. Assuming $|a|\leq1$, we compute
\begin{align*}
&|a|^{-m} |\mathcal S_\psi f((b_1,b_2),s,a)|\\
&=|a|^{\frac{7}{4}}|\int_{\FBR}\int_{\FBR}\xi_1^{m+1}\FBcF f(\xi_1,\xi_2)e^{2\pi i (b_1\xi_1+b_2\xi_2)}\overline{\FBcF\psi_{m}(a\xi_1,a|a|^{-\frac{1}{2}}(\xi_2-s\xi_1))}\D\xi_1\D\xi_2|\\
%&=C|a|^{\frac{3}{4}}|\int_{\FBR}\int_{\FBR}\FBcF (\partial_{\xi_1}^{m+1}f)(\xi_1,\xi_2)e^{2\pi i (b_1\xi_1+b_2\xi_2)}\overline{\FBcF\psi_{m+1}(a\xi_1,a^{\frac{1}{2}}(\xi_2-s\xi_1))}\D\xi_1\D\xi_2|\\
&\lesssim |a|^{\frac{7}{4}}\int_{\FBR}\int_{\FBR}|\FBcF (\partial_{1}^{m+1}f)(\xi_1,\xi_2)||\FBcF\psi_{m}(a\xi_1,a|a|^{-\frac{1}{2}}(\xi_2-s\xi_1))|\D\xi_1\D\xi_2\\
&\lesssim |a|^{\frac{7}{4}}\|\partial_{1}^{m+1}f\|_1 \int_{\FBR}\int_{\FBR} |\FBcF\psi_{m}(a\xi_1,a|a|^{-\frac{1}{2}}(\xi_2-s\xi_1))|\D\xi_1\D\xi_2\\
&\lesssim|a|^{\frac{1}{4}}\|\partial_{1}^{m+1}f\|_1\|\FBcF\psi_{m}\|_1
\lesssim\rho_{{m+1}+2}(f)\rho_{2}(\FBcF\psi_{m}).
\end{align*}
By Lemma~\ref{FBLm:mainlm}, and since the Fourier transform and the differentiation are continuous operators from $\mathcal{S}_{0} (\mathbb{R}^2)$ into itself, we obtain
\begin{align}\label{eq:inequalitynegativepowera}
|a|^{-m} |\mathcal S_\psi f((b_1,b_2),s,a)|\lesssim \rho_{\nu_1}(f)\rho_{\nu_2}(\psi)
\end{align}
for some $\nu_1,\, \nu_2\in\mathbb{N}$ and for every $(b_1,b_2)\in\FBR^2,\, s\in\FBR$ and $|a|\leq1$.
\par
Equation~\eqref{eq:inequalitypositivepowera} together with \eqref{eq:inequalitynegativepowera} show that for every $f, \psi\in \mathcal{S}_{0} (\mathbb{R}^2)$, given $m\in \mathbb{N}$, there exist $ \nu_1,\nu_2 \in \FBN$ such that
\begin{equation*}
\rho_{0,0,0,m} ^{0,0,0,0} ( \mathcal S_\psi f) \lesssim \rho_{\nu_1} (f )\rho_{\nu_2} (\psi)
\end{equation*}
and this concludes the proof.
\end{proof}

Theorem  \ref{thm:continuityshearlettransform} shows how the regularity and decay properties in the range of the shearlet transform are controlled by the corresponding properties of both the shearlet  $\psi $ and the analyzed signal $f.$

\section{The Shearlet Synthesis Operator}%%%%%%%%%%%%%%%%%%%%%%%%%%%%%%THESHEARLESYNTHESISOPERATOR
\label{FBsec:4}

Reconstruction formula~\eqref{FBreconstructionformulashearlet}  suggests to define a linear operator which maps functions over $\mathbb{S}$ to functions over the Euclidean plane $\FBR^2$. The operator that we are after is
analogous to the wavelet synthesis operator considered by Holschneider \cite{FBhol1995}.
\begin{definition}\label{defn:synthesisoperator}
Let $\psi\in L^1(\FBR^2)$ and let $F\in \FBcS(\mathbb{S})$. We
define the shearlet synthesis operator $\FBcS_{\psi}^tF$ of $F$ with respect to $\psi$ as the function
\begin{align}\label{FBsynthesisoperator}
&\FBcS_{\psi}^t F(x)=\int_{\FBR^\times}\int_{\FBR}\int_{\FBR^2} F(b,s,a)\,\, \pi_{b,s,a}\psi(x)
\,\,\frac{\D b \D s  \D a}{|a|^3},\qquad x\in\FBR^2.
\end{align}
\end{definition}

It is immediate to verify that for every $F\in \FBcS(\mathbb{S})$ the integral in \eqref{FBsynthesisoperator} is absolutely convergent for all $x\in\FBR^2$. Moreover, $\FBcS ^t: (F, \psi) \mapsto \FBcS_{\psi}^t F$ is a continuous bilinear mapping from
$  \FBcS(\mathbb{S})  \times L^1 (\FBR^2)$ into $ L^\infty(\mathbb R^2)$, which is proved in what follows.
Let $F\in \FBcS(\mathbb{S})$, $x\in\FBR^2$, and $\epsilon>0$. Indeed, by the Fubini theorem and the definition of the shearlet representation \eqref{eq:shearletrappr} it follows that
\begin{align*}
&|\FBcS_{\psi}^t F(x)|\leq\int_{|a|<\epsilon}\int_{\FBR}\int_{\FBR^2} |a|^{-\frac{3}{4}} |F(b,s,a)|\,\, |\psi(A_a^{-1}N_s^{-1}(x-b))|
\,\,\frac{\D b \D s  \D a}{|a|^3}\\
&+\int_{|a|>\epsilon}\int_{\FBR}\int_{\FBR^2} |a|^{-\frac{3}{4}} |F(b,s,a)|\,\, |\psi(A_a^{-1}N_s^{-1}(x-b))|
\,\,\frac{\D b \D s  \D a}{|a|^3}\\
&\leq\rho_{0,0,2,\frac{9}{4}}^{0,0,0,0}(F)\int_{|a|<\epsilon}\int_{\FBR}\int_{\FBR^2} |a|^{-\frac{3}{2}} |\psi(A_a^{-1}N_s^{-1}(x-b))|
\,\,\frac{\D b \D s  \D a}{(1+s^2)}\\
&+\rho_{0,0,2,0}^{0,0,0,0}(F)\int_{|a|>\epsilon}|a|^{-\frac{9}{4}}\int_{\FBR}\int_{\FBR^2} |a|^{-\frac{3}{2}} |\psi(A_a^{-1}N_s^{-1}(x-b))|
\,\,\frac{\D b \D s  \D a}{(1+s^2)}\\
&\leq\rho_{0,0,2,\frac{9}{4}}^{0,0,0,0}(F)\|\psi\|_1\int_{|a|<\epsilon}\int_{\FBR}
\,\,\frac{ \D s  \D a}{(1+s^2)}
+\rho_{0,0,2,0}^{0,0,0,0}(F)\|\psi\|_1\int_{|a|>\epsilon}|a|^{-\frac{9}{4}}\int_{\FBR}
\,\,\frac{\D s  \D a}{(1+s^2)}\\
&\lesssim\|\psi\|_1(\rho_{0,0,2,\frac{9}{4}}^{0,0,0,0}(F)+\rho_{0,0,2,0}^{0,0,0,0}(F))<+\infty.
\end{align*}

The terminology shearlet synthesis operator follows by an analogy with the frame theory. Indeed, an admissible
shearlet $\psi$  for the shearlet representation~\eqref{FBeqn:admvect} gives rise to
the continuous frame $\{\pi_{b,s,a}\psi\}_{(b,s,a)\in\mathbb{S}}$, see e.g. \cite{FBgr11},
and the frame synthesis operator associated to such a frame is actually the shearlet synthesis operator $\mathcal{S}^t_\psi$ given by
\eqref{FBsynthesisoperator}.

\par

We consider the shearlet synthesis operator as the bilinear map $\mathcal S^t\colon (F, \psi) \mapsto \mathcal S^t_\psi F$
for a wide class of pairs $(F, \psi)$ for which Definition~\ref{defn:synthesisoperator} makes sense,
and prove its continuity from $\mathcal{S}(\mathbb{S})\times\mathcal{S}_0(\FBR^2)$ into $\mathcal{S}_0 (\mathbb{R}^2)$.
We refer to \cite[Theorem 19.0.1]{FBhol1995} for the analogous statement for the wavelet synthesis operator.

\begin{theorem}
The operator $\mathcal S^t\colon (F, \psi) \mapsto \mathcal S_\psi F$ is a continuous bilinear mapping from $\mathcal{S}(\mathbb{S})\times\mathcal{S}_0(\FBR^2)$ into $\mathcal{S}_0 (\mathbb{R}^2)$.
In particular, the shearlet synthesis operator $S^t_\psi$ is a continuous operator from $\mathcal{S}(\mathbb{S})$ into $\mathcal{S}_0 (\mathbb{R}^2)$ for every $\psi\in\mathcal{S}_0 (\mathbb{R}^2)$.
 \label{FBthm:continuitysynthesisoperator}
\end{theorem}
\begin{proof}
We start by proving the continuity, and we refer to \cite{FBbarpilnen} for the analogous proof under stronger conditions on the admissible vectors.
Indeed, compared to \cite{FBbarpilnen}, here we consider general admissible shearlets $\psi\in\mathcal{S}_0 (\mathbb{R}^2)$.

We need to show that for every $F\in\FBcS(\mathbb{S})$, $\psi\in\mathcal{S}_0 (\mathbb{R}^2)$ and $\nu_1 \in \mathbb{N}$, there exist $k_1,k_2,l,m,\alpha_1,\alpha_2,\beta,\gamma,\nu_2\in \FBN$ such that
\begin{equation*}
\rho_{\nu_1}(\FBcS^t_{\psi}F)\leq C \rho_{k_1,k_2,l,m} ^{\alpha_1,\alpha_2, \beta,\gamma} (F)\rho_{\nu_2}(\psi).
\end{equation*}

We use the fact that the families $\hat{\rho}_{\nu}(f)=\rho_{\nu}(\FBcF f)$
and $\hat\rho_{k_1,k_2,l,m} ^{\alpha_1,\alpha_2, \beta,\gamma}(F)=\rho_{k_1,k_2,l,m} ^{\alpha_1,\alpha_2, \beta,\gamma}(\FBcF F)$,
where $\FBcF F$ denotes the Fourier transform of $F$ with respect to the variable $b$, are bases of norms for the topologies of $\FBcS_0(\FBR^2)$ and $\mathcal{S}(\mathbb {S})$, respectively (cf. \cite{FBkpsv2014}). Furthermore, by the Plancherel theorem, equation \eqref{FBshearfreq} and the Fubini theorem we have that
$$
\FBcS_{\psi}^t F(x)
=\int_{\FBR^{\times}}\int_{\FBR} |a|^{3/4}\int_{\FBR^2} \FBcF F(\mathbf{\xi},s,a)\,\, e^{2\pi i x\xi} \overline{\FBcF \psi(-a\xi_1,a|a|^{-\frac{1}{2}}(-\xi_2+s\xi_1))}
\,\,\D \xi\frac{ \D s \D a }{|a|^3}
$$
\begin{equation}
\label{FBfrequencysynthesisoperator}
=\int_{\FBR^2}e^{2\pi i x\xi} \int_{\FBR^{\times}}\int_{\FBR} |a|^{3/4} \FBcF F(\mathbf{\xi},s,a)\,\, \overline{\FBcF \psi(-a\xi_1,a|a|^{-\frac{1}{2}}(-\xi_2+s\xi_1))}
\,\,\frac{ \D s \D a }{|a|^3}\D \xi,
\end{equation}
for every $F\in\FBcS(\mathbb{S})$ and $\psi\in\mathcal{S}_0 (\mathbb{R}^2)$.

\par

We first consider the derivatives of $ \FBcS^t_{\psi}F.$ Let  $\epsilon>0$ and $N\in\mathbb{N}$, $N>2$. By formula \eqref{FBfrequencysynthesisoperator} for every $\alpha\in  \mathbb{N}$ we have that
\begin{align*}
&|\partial_{x_1}^{\alpha}(\FBcS^t_{\psi}F)(x_1,x_2)|\\
&\lesssim\int_{\FBR^2}\int_{\FBR^{\times}}\int_{\FBR} |a|^{\frac{3}{4}} |\xi_1|^\alpha |\FBcF F(\mathbf{\xi},s,a)|\,\, |\FBcF \psi(-A_a{^t\!N_s}\xi)|
\,\,\frac{ \D s \D a \D \xi_1\D \xi_2}{|a|^3}\\
&\lesssim\int_{\FBR^2}\int_{|a|<\epsilon}\int_{\FBR} \langle \xi \rangle^{N}
\langle s \rangle^2 |a|^{-\frac{9}{4}} |\xi_1|^{\alpha}|\FBcF F(\mathbf{\xi},s,a)|\,\,|\FBcF \psi(-A_a{^t\!N_s}\xi)|
\frac{ \D s \D a \D \xi_1\D \xi_2}{\langle s \rangle^2 \langle \xi \rangle^{N} }\\
&+\int_{\FBR^2}\int_{|a|>\epsilon}\int_{\FBR}
\langle \xi \rangle^{N} \langle s \rangle^2
|a|^{-\frac{9}{4}} |\xi_1|^{\alpha}|\FBcF F(\mathbf{\xi},s,a)|\,\,|\FBcF \psi(-A_a{^t\!N_s}\xi)|
\frac{ \D s \D a \D \xi_1\D \xi_2}{\langle s \rangle^2 \langle \xi \rangle^{N} }\\
&\lesssim \rho_{N+\alpha,N,2,\frac{9}{4}} ^{0,0,0,0} (\FBcF F)\rho_0(\psi)+\rho_{N+\alpha,N,2,0} ^{0,0,0,0} (\FBcF F)\rho_0(\psi),
\end{align*}
which is dominated by a single product of norms. The terms $|\partial_{x_2}^{\beta}(S^t_{\psi}\Phi)(x_1,x_2)|$, $\beta\in  \mathbb{N}$, can be estimated in a similar fashion.

Next we consider multiplications by  monomials $x_1 ^k$,  $k\in  \mathbb{Z}_+$.
By formula \eqref{FBfrequencysynthesisoperator} we have
\begin{align*}
&|x_1^{k}(\FBcS^t_{\psi}F)(x_1,x_2)|\\
&=|\int_{\FBR^2}x_1^ke^{2\pi i x\xi} \int_{\FBR^{\times}}\int_{\FBR} |a|^{\frac{3}{4}} \FBcF F(\mathbf{\xi},s,a)\,\, \overline{\FBcF \psi(-a\xi_1,a|a|^{-\frac{1}{2}}(-\xi_2+s\xi_1))}
\,\,\frac{ \D s \D a \D \xi_1\D \xi_2}{|a|^3}|\\
&=|\int_{\FBR^2}
\frac{e^{2\pi i x\xi}}{(2\pi i)^{k}}
\int_{\FBR^{\times}}\int_{\FBR} |a|^{\frac{3}{4}} \partial_{\xi_1}^k[\FBcF F(\mathbf{\xi},s,a)\,\, \overline{\FBcF \psi(-a\xi_1,a|a|^{-\frac{1}{2}}(-\xi_2+s\xi_1))}]
\,\,\frac{ \D s \D a \D \xi}{|a|^3}|\\
&\lesssim \int_{\FBR^2}\int_{\FBR^{\times}}\int_{\FBR} |a|^{\frac{3}{4}} |\partial_{\xi_1}^k[\FBcF F(\mathbf{\xi},s,a)\,\, \FBcF \psi(-a\xi_1,a|a|^{-\frac{1}{2}}(-\xi_2+s\xi_1))]|
\,\,\frac{\D s \D a \D \xi}{|a|^3},
\end{align*}
which is less than or equal to a finite sum of addends of the form
\begin{align*}
&\int_{\FBR^2}\int_{\FBR^{\times}}\int_{\FBR} |a|^{\frac{3}{4}+k_2+\frac{k_3}{2}}|s|^{k_3} |\partial_{\xi_1}^{k_1}\FBcF F(\mathbf{\xi},s,a)| \times \\
& \times |\partial_1^{k_2}\partial_2^{k_3}\FBcF \psi(-a\xi_1,a|a|^{-\frac{1}{2}}(-\xi_2+s\xi_1))|
\frac{ \D s \D a \D \xi_1\D \xi_2}{|a|^3}\\
&=\int_{\FBR^2}\int_{\FBR^{\times}}\int_{\FBR} |a|^{-\frac{9}{4}+k_2+\frac{k_3}{2}}|s|^{k_3}(1+|\xi|^2)^{N/2}(1+s^2)^{\frac{M}{2}}|\partial_{\xi_1}^{k_1}\FBcF F(\mathbf{\xi},s,a)|\times\\
& \times|\partial_1^{k_2}\partial_2^{k_3}\FBcF \psi(-a\xi_1,a|a|^{-\frac{1}{2}}(-\xi_2+s\xi_1))|
\frac{ \D s \D a \D \xi_1\D \xi_2}{(1+s^2)^{\frac{M}{2}}(1+|\xi|^2)^{\frac{N}{2}}},
\end{align*}
where $k_1,\, k_2,\, k_3\in\FBN$ are less than $k$, $N>2$ and $M> k_3+1$.
Then, splitting the integral over $\FBR^\times$ into integrals over
$|a|<\epsilon$ and $|a|>\epsilon$, with $\epsilon>0$, we obtain
\begin{align*}
&\int_{\FBR^2}\int_{\FBR^{\times}}\int_{\FBR} |a|^{\frac{3}{4}+k_2+\frac{k_3}{2}}|s|^{k_3} |\partial_{\xi_1}^{k_1}\FBcF F(\mathbf{\xi},s,a)| \times \\
& \times |\partial_1^{k_2}\partial_2^{k_3}\FBcF \psi(-a\xi_1,a|a|^{-\frac{1}{2}}(-\xi_2+s\xi_1))|
\frac{ \D s \D a \D \xi_1\D \xi_2}{|a|^3}\\
&\lesssim \rho_{N,N,M,|-\frac{9}{4}+k_2+\frac{k_3}{2}|} ^{k_1,0,0,0} (\FBcF F)\rho_{k_2+k_3}(\psi)
+\rho_{N,N,M,k_2+\frac{k_3}{2}} ^{k_1,0,0,0} (\FBcF F)\rho_{k_2+k_3}(\psi),
\end{align*}
which is dominated by a single product of norms.

\par

Obviously, we can treat $|{x_2}^{k}(\FBcS^t_{\psi}F)(x_1,x_2)|$, $k\in  \mathbb{Z}_+$, in the same manner,
and we conclude that the operator $\mathcal S^t $ is continuous from
$  \mathcal{S}(\mathbb {S})\times \mathcal{S}_0(\mathbb{R}^2)$ into $ \mathcal{S}(\mathbb{R}^2) $.

\par

Finally, we have to show that $\FBcS^t_{\psi}F\in\FBcS_0(\FBR^2)$ for every $F\in \mathcal{S}(\mathbb {S})$ and $\psi\in\mathcal{S}_0(\mathbb{R}^2)$. The idea is to prove the equivalent condition
\begin{equation}\label{FBeq:limitmoments}
\lim_{\xi\to0}\frac{\FBcF \FBcS^t_{\psi}F(\xi)}{|\xi|^k}=0,
\end{equation}
for every $k\in\FBN$, see {\cite[Lemma 6.0.4]{FBhol1995}}. We first observe that by formula \eqref{FBfrequencysynthesisoperator} and
the Fourier inversion formula we have
\begin{align}\label{synthesisfrequency}
&\FBcF \FBcS_{\psi}^t F(\xi_1,\xi_2)=\int_{\FBR^{\times}} \int_{\FBR} |a|^{3/4} \FBcF F(\mathbf{\xi},s,a)\,\, \overline{\FBcF \psi(-a\xi_1,a|a|^{-\frac{1}{2}}(-\xi_2+s\xi_1))}
\,\,\frac{ \D s \D a }{|a|^3},
\end{align}
where we recall that $\FBcF F$ denotes the Fourier transform of $F$ with respect to the variable $b$. We will prove that for every $k\in\FBN$ there exists $N_k\in\mathbb{Z}_+$ and a constant $C>0$ such that
\[
\frac{|\FBcF \FBcS^t_{\psi}F(r\cos{\theta},r\sin{\theta})|}{r^k}\leq Cr^{N_k},
\]
for every $r\in\FBR_+$ and for every $\theta\in[0,2\pi)$.
Indeed, by equation \eqref{synthesisfrequency}, we have
\begin{align*}
&|\FBcF S^t_{\psi}F(r\cos{\theta},r\sin{\theta})|
\leq  \int_{\FBR^{\times}}\int_{\FBR} |a|^{3/4} |\FBcF F((r\cos{\theta},r\sin{\theta}),s,a)|\times \\
& \times |\FBcF \psi(-ar\cos{\theta},a|a|^{-\frac{1}{2}}r(-\sin{\theta}+s\cos{\theta}))|
\,\,\frac{ \D s \D a }{|a|^3}.
\end{align*}

By Lemma \ref{FBLm:mainlm}, for every $m\in\FBN$ there exists $g\in  \mathcal{S}_{0} (\mathbb{R}^2)$ such that
\begin{align*}
& \FBcF \psi(\xi_1,\xi_2)=\xi_1^m \FBcF g(\xi_1,\xi_2),\quad \xi_1,\,\xi_2\in\FBR
\end{align*}
and we can continue the above inequality as follows
\begin{align*}
&|\FBcF \FBcS^t_{\psi}F(r\cos{\theta},r\sin{\theta})|\\
&\lesssim \int_{\FBR^{\times}}\int_{\FBR} |a|^{-\frac{9}{4}+m} |\FBcF F((r\cos{\theta},r\sin{\theta}),s,a)|r^{m}|\cos{\theta}|^{m}\times\\
&\times |\FBcF g(-ar\cos{\theta},a|a|^{-\frac{1}{2}}r(-\sin{\theta}+s\cos{\theta}))|(1+s^2)\frac{ \D s \D a }{(1+s^2)}.
\end{align*}

Next, by splitting the integral over $\FBR^{\times}$ into integrals ove $|a|<\epsilon$ and $|a|>\epsilon$ for some $\epsilon>0$, we obtain
\begin{align*}
&|\FBcF \FBcS^t_{\psi}F(r\cos{\theta},r\sin{\theta})|
\lesssim r^{m}[\rho_{0,0,2,\frac{9}{4}}^{0,0,0,0}(\FBcF F)\rho_0(g)+\rho_{0,0,2,m}^{0,0,0,0}(\FBcF F)\rho_0(g)]\lesssim r^{m},
\end{align*}
where the hidden constant is independent of $\theta\in[0,2\pi)$. Therefore, by choosing $m>k$, we obtain
 \[
\frac{|\FBcF \FBcS^t_{\psi}F(r\cos{\theta},r\sin{\theta})|}{r^k}\lesssim r^{m-k},
\]
which implies that \eqref{FBeq:limitmoments} holds true for every $k\in\FBN$, and we conclude that $\FBcS^t_{\psi}F$ belongs to $\FBcS_0(\FBR^2)$.
\end{proof}

Next we show that the shearlet synthesis operator is in fact the adjoint of the shearlet transform in the following sense.
Let $\psi\in\mathcal{S}_0(\mathbb{R}^2)$. If $f\in L^1(\FBR^2)\cap L^2(\FBR^2)$ and $F\in \FBcS(\mathbb{S})$, then by Fubini theorem we have that
\begin{align}\label{FBeq:dualityrelationfirst}
\nonumber\int_{\FBR^2}f(x)\FBcS_{\overline{\psi}}^tF(x)\D x&=\int_{\FBR^2}f(x)\int_{\FBR^\times}\int_{\FBR}\int_{\FBR^2} F(b,s,a)
\,\, \overline{\pi_{b,s,a}\psi(x)}\,\,\frac{\D b\D s\D a}{|a|^3}\D x\\
\nonumber&=\int_{\FBR^\times}\int_{\FBR}\int_{\FBR^2}F(b,s,a)\int_{\FBR^2}f(x)\overline{\pi_{b,s,a}\psi(x)}\D x
\,\,\frac{\D b\D s\D a}{|a|^3}\\
&=\int_{\FBR^\times}\int_{\FBR}\int_{\FBR^2}\FBcS_{\psi}f(b,s,a)F(b,s,a)\,\,\frac{\D b\D s\D a}{|a|^3}.
\end{align}
Therefore, since $L^1(\FBR^2)\cap L^2(\FBR^2)$ naturally embeds into $\FBcS'_0(\FBR^2)$ and by the identification~\eqref{FBdualitysynthesisspace}, we may write \eqref{FBeq:dualityrelationfirst} as
\begin{equation*}
_{\FBcS'(\mathbb{S})}(\FBcS_{\psi}f,F )_{ \FBcS(\mathbb{S})}={_{\FBcS'_0(\FBR^2)}}(f, \FBcS_{\overline{\psi}}^tF )_{\FBcS_0(\FBR^2)}.
\end{equation*}

This duality relation will motivate our definition of the distributional shearlet transform in Section \ref{FBsec:5}.

\par

If $f\in L^2(\FBR^2)$ then, by using the definition of the shearlet synthesis operator, we can rewrite reconstruction formula~\eqref{FBreconstructionformulashearlet2} as
\begin{equation}\label{eq:newexpressionreconstructionformula}
f= \frac{1}{C_{\psi,\phi}} (\FBcS_{\phi}^t\circ\FBcS_{\psi})f,
\end{equation}
where the equality holds in the weak-sense. Next we show that if $\psi,\, \phi,\, f\in \FBcS_0(\FBR^2)$, then the reconstruction formula holds pointwise.

\begin{proposition}\label{prop:FBreconstructionformulashearletpointwise}
Let $\psi,\, \phi\in \FBcS_0(\FBR^2)$ be such that \eqref{FBeqn:admvectequivalent2} holds true. If $f\in \FBcS_0(\FBR^2)$, then
\begin{equation}\label{FBreconstructionformulashearletpointwise}
 f(x) = \frac{1}{C_{\psi,\phi}}  (\FBcS_{\phi}^t\circ\FBcS_{\psi})f(x), \;\;\; x\in \FBR^2.
\end{equation}
\end{proposition}

\begin{proof}
We recall that by Theorem~\ref{thm:continuityshearlettransform} if $\psi$ and $f$ belong to $\FBcS_0(\FBR^2)$, then $\FBcS_{\psi} f\in\FBcS(\mathbb{S})$. By the Plancherel theorem, Fubini theorem and formula~\eqref{FBshearfreq} for every $x\in \FBR^2$ we have
\begin{align*}
 (\FBcS_{\phi}^t\circ\FBcS_{\psi})f(x)&=\int_{\FBR^{\times}}\int_{\FBR}\int_{\FBR^2}
 \FBcS_{\psi}f(b,s,a) \,     \pi_{b,s,a}\phi(x)\ \frac{\D b\D s\D a}{|a|^3}\\
 &=\int_{\FBR^{\times}}\int_{\FBR}\int_{\FBR^2}\int_{\FBR^2}\FBcF f(\xi) e^{2\pi i b \xi} \overline{\FBcF \psi(A_a{^t\!N_s}\xi)}
\, \pi_{b,s,a}\phi(x)\ \frac{\D\xi\D b\D s\D a}{|a|^\frac{9}{4}}\\
 &=\int_{\FBR^{\times}}\int_{\FBR}\int_{\FBR^2}\FBcF f(\xi)\overline{\FBcF \psi(A_a{^t\!N_s}\xi)}\int_{\FBR^2} e^{2\pi i b\xi}
\,  \pi_{b,s,a}\phi(x)\D b\ \frac{\D \xi\D s\D a}{|a|^\frac{9}{4}}\\
 &=\int_{\FBR^{\times}}\int_{\FBR}\int_{\FBR^2}\FBcF f(\xi)\overline{\FBcF \psi(A_a{^t\!N_s}\xi)}\FBcF\pi_{-x,s,a}\phi(\xi)\ \frac{\D \xi\D s\D a}{|a|^\frac{9}{4}}\\
 &=\int_{\FBR^{\times}}\int_{\FBR}\int_{\FBR^2}\FBcF f(\xi)\overline{\FBcF \psi(A_a{^t\!N_s}\xi)}\FBcF \phi(A_a{^t\!N_s}\xi)e^{2\pi i x \xi}\ \frac{\D \xi\D s\D a}{|a|^\frac{3}{2}}.
\end{align*}
Hence, by equation~\eqref{FBeqn:admvectequivalent2} for every $x\in \FBR^2$
\begin{align*}
 (\FBcS_{\phi}^t\circ\FBcS_{\psi})f(x)=\int_{\FBR^2}\FBcF f(\xi)e^{2\pi i x\xi}\int_{\FBR^{\times}}\int_{\FBR}\overline{\FBcF \psi(A_a{^t\!N_s}\xi)}\FBcF \phi(A_a{^t\!N_s}\xi)\ \frac{\D s\D a}{|a|^\frac{3}{2}}\D \xi= C_{\psi,\phi} f(x),
\end{align*}
and this concludes the proof.
\end{proof}

\section{The shearlet transform on $\mathcal{S}'_0(\FBR^2)$}%%%%%%%%%%%%%%%%%%%%%%%%%%%%%%	THE SHEARLET TRANSFORM OF DISTRIBUTIONS
\label{FBsec:5}

In this last section, we extend the definition of the shearlet transform to the space of Lizorkin distributions and we prove its consistency with the classical definition
for test functions. Then we extend reconstruction formula~\eqref{eq:newexpressionreconstructionformula} to $\mathcal{S}'_0(\FBR^2)$.
\begin{definition}\label{FBdefn:shearlettransformdistributions}
Let $f\in \FBcS'_0(\FBR^2)$ and $\psi\in\mathcal{S}_0(\FBR^2)$. We define the shearlet transform $\mathscr{S}_\psi f$ of $f$ with respect to $\psi$ as the distribution in $\FBcS'(\mathbb{S})$ whose action
on test functions is given by
\begin{equation*}
_{\FBcS'(\mathbb{S})}( \mathscr S_{\psi}f,\Phi )_{\FBcS(\mathbb{S})} = {_{\FBcS'_0(\FBR^2)}}( f,\FBcS^t_{\overline{\psi}}\Phi )_{\FBcS_0(\FBR^2)},\quad \Phi\in \FBcS(\mathbb{S}).
\end{equation*}
\end{definition}

The validity of this definition follows from Theorem~\ref{FBthm:continuitysynthesisoperator}. Moreover, Definition \ref{FBdefn:shearlettransformdistributions}
and Theorem \ref{thm:continuityshearlettransform} motivate the next one.

\begin{definition}\label{FBdefn:synthesisoperatordistributions}
Let $\Psi\in \FBcS'(\mathbb{S})$ and $\psi\in\mathcal{S}_0(\FBR^2)$. The shearlet synthesis operator $\mathscr S^t_{\psi}\Psi$ of $\Psi$ with respect to $\psi$ is defined as the
Lizorkin distribution whose action on test functions is given by
\[
{_{\FBcS'_0(\FBR^2)}}(\mathscr S^t_{\psi}\Psi,f)_{\FBcS_0(\FBR^2)}=
{_{\FBcS'(\mathbb{S})}}( \Psi,\FBcS_{\overline{\psi}}f)_{\FBcS(\mathbb{S})},\quad f\in \FBcS_0(\FBR^2).
\]
\end{definition}

The consistence of Definitions~\ref{FBdefn:shearlettransformdistributions} and \ref{FBdefn:synthesisoperatordistributions} is guaranteed by Theorems~\ref{FBthm:continuitysynthesisoperator} and \ref{thm:continuityshearlettransform}, respectively. Furthermore, we immediately obtain the following result.
\begin{proposition}
The shearlet transform $\mathscr S_{\psi}\colon \FBcS'_0(\FBR^2)\to\FBcS'(\mathbb{S})$ and the shearlet synthesis operator $\mathscr S^t_{\psi}: \FBcS'(\mathbb{S})\to \FBcS'_0(\FBR^2)$ are continuous linear operators.
\end{proposition}
\begin{proof}
The proof follows immediately by Theorem~\ref{thm:continuityshearlettransform} and \ref{FBthm:continuitysynthesisoperator}.
\end{proof}
We now extend the reconstruction formula \eqref{eq:newexpressionreconstructionformula} to the space of Lizorkin distributions.
\begin{proposition}\label{prop:reconstructionformuladistributions}
Let $\psi,\,\phi\in \FBcS_0(\FBR^2)$ be such that \eqref{FBeqn:admvectequivalent2} holds true. For every $f\in\FBcS'_0(\FBR^2)$
\[
f= \frac{1}{C_{\psi,\phi}}(\mathscr S_{\phi}^t\circ\mathscr S_{\psi})f.
\]
\end{proposition}
\begin{proof}
Let $f\in \FBcS'_0(\FBR^2)$. By Definitions~\ref{FBdefn:shearlettransformdistributions}, \ref{FBdefn:synthesisoperatordistributions} and by Proposition~\ref{prop:FBreconstructionformulashearletpointwise} we have that for every $\varphi\in \FBcS_0(\FBR^2)$ the following equalities hold:
\begin{align*}
((\mathscr S_{\phi}^t\circ\mathscr S_{\psi})f,\varphi)=(\mathscr S_{\psi}f,\mathcal S_{\overline{\phi}}\varphi)=( f,(\mathcal S^t_{\overline{\psi}}\circ\mathcal S_{\overline{\phi}})\varphi)= C_{\overline{\phi},\overline{\psi}}( f,\varphi)
=C_{\psi,\phi}( f,\varphi),
\end{align*}
as claimed.
\end{proof}

The next theorem shows that Definition~\ref{FBdefn:shearlettransformdistributions} is in fact
consistent with the shearlet transform for test functions introduced in Definition \ref{FBdefsheartransf}.
Furthermore, it also shows that our definition is a natural extension of  the one considered in \cite{FBkula09,FBgr11},
where the shearlet transform of a tempered distribution $f$ with respect to an admissible vector $\psi\in\FBcS(\FBR^2)$ is defined as the function on $\mathbb{S}$ given by
\begin{equation*}
S_\psi f(b,s,a)={_{\mathcal{S}'(\FBR^2)}}(f, \pi_{b,s,a}\psi)_{\mathcal{S}(\FBR^2)},
\end{equation*}
for every $(b,s,a)\in\mathbb{S}$.
Following the coorbit space approach, given a suitable test function space usually denoted by $\mathcal{H}_{1,w}$, where $w$ is a weight function,
and its anti-dual $\mathcal{H}_{1,w}^{\sim}$, the extended shearlet transform of $f\in \mathcal{H}_{1,w}^{\sim}$ with respect to $\psi\in \mathcal{H}_{1,w}$ is defined by
\[
\FBcS_{\psi}f(b,s,a)={_{\mathcal{H}_{1,w}^{\sim}}}(f, \pi_{b,s,a}\psi)_{\mathcal{H}_{1,w}},
\]
for every $(b,s,a)\in\mathbb{S}$, \cite{dahlikeetal}. Theorem~\ref{FBteo:desingularization} shows the equivalence of our duality approach with the coorbit space one
in the sense that it identifies the shearlet transform of any Lizorkin distribution with  the function given by
\[
(b,s,a)\mapsto{_{\FBcS'_0(\FBR^2)}}( f, \pi_{b,s,a}\psi)_{\FBcS_0(\FBR^2)},
\;\;\;  (b,s,a)\in\mathbb{S}.
\]

\par

\begin{theorem}\label{FBteo:desingularization}
Let $\psi\in\mathcal{S}_0(\FBR^2)$ and $f\in \FBcS_0'(\FBR^2)$. The function $S_\psi f$ defined by
\begin{equation}\label{eq:functiondesingularization}
S_\psi f(b,s,a)={_{\mathcal{S}'_0(\FBR^2)}}(f, \pi_{b,s,a}\psi)_{\mathcal{S}_0(\FBR^2)}
\end{equation}
is a smooth function of at most polynomial growth on $\mathbb{S}$.
Furthermore, by identification~\eqref{FBdualitysynthesisspace} it follows that
%\begin{equation*}
%\FBcS_{\psi}f(b,s,a)=\langle f, S_{b,s,a}\psi\rangle,
%\end{equation*}
%that is,
\begin{equation}\label{eq:identificationdistributionfunction}
_{\FBcS'(\mathbb{S})}( \mathscr S_{\psi}f,\Phi )_{\FBcS(\mathbb{S})}={_{\FBcS'(\mathbb{S})}}(S_\psi f, \Phi)_{\FBcS(\mathbb{S})},
\;\;\; \text{ for every } \;\;\; \Phi\in\FBcS(\mathbb{S}).
%\int_{\mathbb{S}} ( f, S_{b,s,a}\psi)_{\mathcal{S}'_0(\FBR^2)\times\mathcal{S}_0(\FBR^2)}\Phi(b,s,a) \frac{\D b\D s\D a}{|a|^3},
\end{equation}
\end{theorem}
\begin{proof}
Consider $f\in\FBcS'_0(\FBR^2)$ and $\psi\in\FBcS_0(\FBR^2)$. Since
the derivatives of $ \pi_{b,s,a}\psi $ are still in $\FBcS_0(\FBR^2)$ it immediately follows that the function given by \eqref{eq:functiondesingularization} is smooth.
We prove that $S_{\psi}f(b,s,a)$ is a function of at most polynomial growth on $\mathbb{S}$ dividing the proof in three steps.

\par

We start by considering $S_\psi f(b,0,1)$ for every $b\in\FBR^2$. Since $f\in\FBcS'_0(\FBR^2)$, it follows that there exists $\nu\in\FBN$ such that
\begin{align*}
|S_{\psi}f(b,0,1)|=|( f,\pi_{b,0,1}\psi)|&\lesssim \sup_{x\in\FBR^2,\, |m|\leq\nu} \langle x\rangle^\nu|\partial^m(\pi_{b,0,1}\psi)(x)|\\
&=\sup_{x\in\FBR^2,\, |m|\leq\nu} \langle x\rangle^\nu|\partial^m\psi(x-b)|.
\end{align*}
Then, by Peetre's inequality we have that
\begin{align*}
|S_{\psi}f(b,0,1)|\lesssim \sup_{x\in\FBR^2,\, |m|\leq\nu} \langle x\rangle^\nu \langle x-b\rangle^{-\nu}\lesssim  \langle b\rangle^{\nu}
\end{align*}
and $S_{\psi}f(b,0,1)$ is of at most  polynomial growth in $b\in\FBR^2$.

\par

We now consider $S_{\psi}f(0,s,1)$ for every $s\in\FBR$. Since $f\in\FBcS'_0(\FBR^2)$, it follows that there exists $\nu\in\FBN$ such that
\begin{align*}
&|S_{\psi}f(0,s,1)|=|(f,\pi_{0,s,1}\psi)|\lesssim \sup_{x\in\FBR^2,\, |m|\leq\nu} \langle x\rangle^\nu|\partial^m(\pi_{0,s,1}\psi)(x)|\\
&=\sup_{x\in\FBR^2,\, |m|\leq\nu} \langle x\rangle^\nu|\sum_{i=0}^{p_m}c_is^i(\pi_{0,s,1}\partial^m\psi)(x)|\leq \sup_{x\in\FBR^2,\, |m|\leq\nu} \langle x\rangle^\nu\sum_{i=0}^{p_m}|c_i||s|^i||\partial^m\psi(N_s^{-1}x)|\\
&\lesssim \sup_{x\in\FBR^2,\, |m|\leq\nu} \langle x\rangle^\nu\sum_{i=0}^{p_m}|c_i||s|^i(1+|N_s^{-1}x|^2)^{-\frac{\nu}{2}}.
\end{align*}

We recall that, for any $x\in\FBR^2$ and $M\in GL(2,\FBR)$
\[
|M^{-1}x|\geq\|M\|^{-1}|x|,
\]
where $\|M\|$ denotes the spectral norm of the matrix $M$. Thus, since $\|N_s\|=(1+s^2/2+(s^2+s^2/2)^{1/2})^{1/2}$ for all $s\in\FBR$, we have that
\begin{align*}
|S_{\psi}f(0,s,1)|
%\sup_{x\in\FBR^2,\, |m|\leq\nu} \langle x\rangle^\nu\sum_{i=0}^{p_m}|c_i||s|^i(1+(\|N_s\|^{-1}|x|)^2)^{-\frac{\nu}{2}}\\
&\lesssim \sup_{x\in\FBR^2,\, |m|\leq\nu} \langle x\rangle^\nu\sum_{i=0}^{p_m}|c_i||s|^i\langle\|N_s\|^{-1}x\rangle^{-\nu}\\
&\lesssim \sup_{x\in\FBR^2,\, |m|\leq\nu} \langle x\rangle^\nu\sum_{i=0}^{p_m}|c_i||s|^i\|N_s\|^{\nu}\langle x\rangle^{-\nu}=\sup_{|m|\leq\nu}\sum_{i=0}^{p_m}|c_i||s|^i\|N_s\|^{\nu},
\end{align*}
which proves that the function $S_{\psi}f(0,s,1)$ is of at most polynomial growth in $s\in\FBR$.

\par

Finally, we consider $S_{\psi}f(0,0,a)$ for every $a\in\FBR^\times$. Since $f\in\FBcS'_0(\FBR^2)$, it follows that there exists $\nu\in\FBN$ such that
\begin{align*}
|S_{\psi}f(0,0,a)|&=|( f,\pi_{0,0,a}\psi)|\lesssim \sup_{x\in\FBR^2,\, |m|\leq\nu} \langle x\rangle^\nu|\partial^m(\pi_{0,0,a}\psi)(x)|\\
&=\sup_{x\in\FBR^2,\, |m|\leq\nu} \langle x\rangle^\nu |a|^{-m_1-\frac{m_2}{2}-\frac{3}{4}}|\pi_{0,0,a}\partial^m\psi(A_a^{-1}x)|\\
%&\lesssim  \sup_{x\in\FBR^2,\, |m|\leq\nu} \langle x\rangle^\nu |a|^{-m_1-\frac{m_2}{2}-\frac{3}{4}}(1+|A_a^{-1}x|^2)^{-\frac{\nu}{2}}\\
&\lesssim  \sup_{x\in\FBR^2,\, |m|\leq\nu} \langle x\rangle^\nu |a|^{-m_1-\frac{m_2}{2}-\frac{3}{4}}\langle\|A_a\|^{-1}x\rangle^{-\nu},
\end{align*}
where $(m_1,m_2)\in\FBN^2$ and $m_1+m_2=m$. Since $\|A_a\|=|a|^{\frac{1}{2}}$ for $|a|<1$, then
\begin{align*}
&|S_{\psi}f(0,0,a)|\lesssim  \sup_{x\in\FBR^2,\, |m|\leq\nu} \langle x\rangle^\nu |a|^{-m_1-\frac{m_2}{2}-\frac{3}{4}} \langle x\rangle^{-\nu}\lesssim a^{-p},
\end{align*}
for some $p\in\FBN$. If $|a|\geq1$, then $\|A_a\|=|a|$ and we have
\begin{align*}
&|S_{\psi}f(0,0,a)|
%\sup_{x\in\FBR^2,\, |m|\leq\nu} \langle x\rangle^\nu a^{-m_1-\frac{m_2}{2}-\frac{3}{4}}(1+(\|A_a\|^{-1}|x|)^2)^{-\frac{\nu}{2}}\\
\lesssim  \sup_{x\in\FBR^2,\, |m|\leq\nu} \langle x\rangle^\nu |a|^{-m_1-\frac{m_2}{2}-\frac{3}{4}+\nu}\langle x\rangle^{-\nu}\lesssim a^{\nu},
\end{align*}
which proves that the function $S_{\psi}f(0,0,a)$ is of at most polynomial growth in $a$ on $\FBR^\times$.

\par

Therefore, since
\[
S_{\psi}f(b,s,a)=(f, \pi_{b,0,1}\pi_{0,s,1}\pi_{0,0,a}\psi),
\]
we conclude that there exist $C,\nu_1,\nu_2,\nu_3>0$ such that
\begin{equation*}
|S_{\psi}f(b,s,a)|\leq C \langle b \rangle^{\nu_1} \langle s \rangle^{\nu_2} \left(a^{\nu_3}+\frac
{1}{a^{\nu_3}}\right),
\end{equation*}
for every $(b,s,a)\in \mathbb{S}$, and we finally identify $S_{\psi}f$ as an element in $\FBcS'(\mathbb{S})$ by formula~\eqref{FBdualitysynthesisspace}.

\par

To prove~\eqref{eq:identificationdistributionfunction} we use the fact that
the space of Lizorkin distributions $\FBcS'_0(\FBR^2)$ is canonically isomorphic
to the quotient of $\FBcS'(\FBR^2)$ by the space of polynomials. By
the Schwartz structural theorem \cite[Theor\'em VI]{FBLS1951},
we can write $f=\partial^{\alpha}g+p$, where $g$ is a continuous slowly growing function on $\FBR^2$, $\alpha\in\FBN^2$ and $p$ is a polynomial. Then it can be easily shown that
$$%\begin{align*}
(\partial^{\alpha}g,\mathcal{S}_{\overline{\psi}}^t\Phi)
%&=(-1)^{|\alpha|}\int_{\FBR^2} g(x) \partial^{\alpha}\mathcal{S}_{\overline{\psi}}^t\Phi(x)\D x\\
%&=(-1)^{|\alpha|} \int_{\FBR^2} g(x) \int_{\mathbb{S}} \Phi(b,s,a)\,\, (\partial^{\alpha}\pi_{b,s,a}\psi)(x)
%\,\,{\rm d}\mu(b,s,a) \D x\\
%&=\int_{\mathbb{S}} \Phi(b,s,a)\,\,(-1)^{|\alpha|}\int_{\FBR^2} g(x)(\partial^{\alpha}\pi_{b,s,a}\psi)(x)\D x {\rm d}\mu(b,s,a)\\
%&=\int_{\mathbb{S}} \Phi(b,s,a)\,\,(-1)^{|\alpha|}\langle g,({S}_{b,s,a}\psi)^{(\alpha)}\rangle {\rm d}\mu(b,s,a)\\
=\int_{\mathbb{S}} \Phi(b,s,a)\,\,(\partial^{\alpha}g,\pi_{b,s,a}\psi) {\rm d}\mu(b,s,a),  \;\;\; \Phi\in\FBcS(\mathbb{S}),
$$%\end{align*}
and analogously,
$$%\begin{align*}
( p,\mathcal{S}_{\overline{\psi}}^t\Phi)= \int_{\mathbb{S}} \Phi(b,s,a)\,\,( p,\pi_{b,s,a}\psi) {\rm d}\mu(b,s,a), \;\;\; \Phi\in\FBcS(\mathbb{S}).
$$%\end{align*}
Then, for $f \in \FBcS'_0(\FBR^2)$  represented by $f=\partial^{\alpha}g+p$ we obtain
\begin{align*}
( f,\mathcal{S}_{\overline{\psi}}^t\Phi) &=( \partial^{\alpha}g+p,\mathcal{S}_{\overline{\psi}}^t\Phi)
=\int_{\mathbb{S}} \Phi(b,s,a)\,\,( \partial^{\alpha}g+p,\pi_{b,s,a}\psi) {\rm d}\mu(b,s,a)\\
&=\int_{\mathbb{S}} \Phi(b,s,a)\,\,( f,\pi_{b,s,a}\psi) {\rm d}\mu(b,s,a)\\
&=\int_{\mathbb{S}} S_{\psi}f(b,s,a)\,\, \Phi(b,s,a) {\rm d}\mu(b,s,a),
\end{align*}
which concludes the proof.
\end{proof}

We finish the paper by showing that the shearlet transform may represent the action of a Lizorkin distribution on any test function in the form
of an absolutely convergent integral over $\FBR^2\times\FBR\times\FBR^\times$. This is sometimes called {\em desingularization}, we refer to \cite{FBkpsv2014}  for details,
see also \cite[Theorem 24.1.4, Chapter I]{FBhol1995} for analogous statement related to the distributional wavelet transform.

\begin{proposition}\label{cor:desingularization}
Let $\psi,\,\phi\in \FBcS_0(\FBR^2)$ be such that \eqref{FBeqn:admvectequivalent2} holds true and $f\in \FBcS_0'(\FBR^2)$. Then, for every $\varphi\in\FBcS_0(\FBR^2)$
\begin{equation}\label{eq:desingularizationtheorem}
( f,\varphi)= \frac{1}{C_{\psi,\phi}} \int_{\FBR^\times}\int_{\FBR}\int_{\FBR^2} S_{\psi}f(b,s,a)\,\,\FBcS_{\overline{\phi}}\varphi(b,s,a)\frac{\D b \D s  \D a}{|a|^3}.
\end{equation}
\end{proposition}

\begin{proof}
Let $\psi,\, \phi\in \FBcS_0(\FBR^2)$ satisfy \eqref{FBeqn:admvectequivalent2} and let $f\in \FBcS_0'(\FBR^2)$.
By Proposition~\ref{prop:reconstructionformuladistributions}, formula~\eqref{eq:identificationdistributionfunction}, and the identification~\eqref{FBdualitysynthesisspace}, it follows that
\begin{align*}
( f,\varphi)=\frac{1}{C_{\psi,\phi}} ((\mathscr S_{\phi}^t\circ\mathscr S_{\psi})f,\varphi)
= \frac{1}{C_{\psi,\phi}} (\mathscr S_{\psi}f,\FBcS_{\overline{\phi}}\varphi)
= \frac{1}{C_{\psi,\phi}} (S_{\psi}f, \FBcS_{\overline{\phi}}\varphi),
\;\;\; \varphi\in\FBcS_0(\FBR^2),
\end{align*}
where $S_{\psi}f$ is the function on $\mathbb{S}$ defined by \eqref{eq:functiondesingularization} and we conclude that formula~\eqref{eq:desingularizationtheorem} holds true.
\end{proof}

%\begin{acknowledgement}
%F. Bartolucci is part of the Computational Harmonic Analysis \& Machine Learning unit of the Machine Learning Genoa Center (MalGa). S. Pilipovi\'{c} and N. Teofanov were supported by the
%Ministry of Education, Science and Technological Development of the
%Republic of Serbia through Project 174024.
%\end{acknowledgement}

%\input{referenc}

\begin{thebibliography}{99.}%
	% Use \bibitem to create references. Do not use links to extermal files, like a bbl-file.

%\bibitem{FBbartoluccithesis} {\sc F. Bartolucci.} {R}adon transforms: Unitarization, Inversion and Wavefront sets [Ph.D. thesis].
%Department of Mathematics, University of Genova (2020). \url{http://hdl.handle.net/11567/997903}.

\bibitem{FBbardemadeviodo}  {\sc F. Bartolucci, F. De Mari, E. De Vito, F. Odone.}
The {R}adon transform intertwines wavelets and shearlets.
\textit{Applied and Computational Harmonic Analysis} \textbf{47} (2019), no. 3,
822--847.

\bibitem{FBbarpilnen} {\sc F. Bartolucci, S. Pilipovi{\'c}, N. Teofanov.}
The shearlet transform and Lizorkin spaces,  	arXiv:2003.06642.

\bibitem{FBcado99} {\sc E.~J. Cand{\`e}s, D.~L. Donoho.} Ridgelets: A key to higher-dimensional intermittency?
\textit{Philos. Trans. R. Soc. A} \textbf{357} (1999), no. 1760, 2495--2509.

\bibitem{FBddgl15} {\sc S. Dahlke, F. De Mari, P. Grohs, D. Labate.}
Harmonic and applied analysis, Appl. Numer. Harmon. Anal.
%Birkh\"auser/Springer, Cham (2015)

\bibitem{FBdahlke2008}
{\sc S.~Dahlke, G.~Kutyniok, P.~Maass, C.~Sagiv, H.~Stark, and G.~Teschke},
 The uncertainty principle associated with the continuous shearlet
  transform, \textit{International Journal of Wavelets, Multiresolution and
  Information Processing} \textbf{6} (2008), no. 2, 157--181.


\bibitem{dahlikeetal}
{\sc S.~Dahlke, G.~Kutyniok, G. Steidl, and G.~Teschke},
Shearlet coorbit spaces and associated Banach spaces,
\textit{Applied and Computational Harmonic Analysis} \textbf{27} (2009), no. 2,
195--214.


\bibitem{FBdumo76}
{\sc M. Duflo, C.C. Moore}
On the regular representation of a nonunimodular locally compact group,
\textit{Journal of Functional Analysis} \textbf{21} (1976), no. 2, 209--243.


\bibitem{feichtingergrochenig89}
{\sc H.G.~Feichtinger, K.~Gr\"ochenig},
Banach spaces related to integrable group representations and their
atomic decompositions, I,
\textit{Journal of Functional Analysis},
\textbf{86} (1989), no. 2, 307--340.

\bibitem{feichtgroche89}
{\sc H.G. Feichtinger, K. Gr\"ochenig},
Banach Spaces Related to
Integrable Group Representations and Their Atomic Decompositions. Part II,
\textit{Monatshefte f\"ur Mathematik},
\textbf{108} (1989), no. 2-3, 129--148.


\bibitem{FBfinkkahler19} {\sc T. Fink, U. Kahler.}
A Space-Based Method for the Generation of a Schwartz Function with Infinitely Many Vanishing Moments of Higher Order with Applications in Image Processing,
\textit{Complex Analysis and Operator Theory} \textbf{13} (2019), no. 3, 985--1010.

\bibitem{FBfolland16}
{\sc G.~B. Folland}, {\em A course in abstract harmonic analysis}, Textbooks in
  Mathematics, CRC Press, Boca Raton, FL, 2nd~ed., 2016.

\bibitem{FBgr11}
{\sc P.~Grohs}, Continuous shearlet frames and resolution of the wavefront
  set, \textit{Monatshefte f{\"u}r Mathematik} \textbf{164} (2011), no. 4, 393--426.

\bibitem{FBgr12} {\sc P. Grohs.} Shearlets and microlocal analysis,  in \cite{FBkula12} (2012) 39--67.
 Shearlets: Multiscale Analysis for Multivariate Data (G. Kutyniok, D. Labate eds.), Birkh\"auser/Springer, 2012.

\bibitem{FBhelgason99}
{\sc S.~Helgason}, {\em The {R}adon transform}, vol.~5 of Progress in
  Mathematics, Birkh\"auser Boston, Inc., Boston, MA, 2nd~ed., 1999.

\bibitem{FBhertle83} {\sc A. Hertle.} Continuity of the Radon transform and its inverse on Euclidean spaces. \textit{Math Z.} \textbf{184} (1983), 165--192.

\bibitem{FBhol1995} {\sc M. Holschneider.} \textit{Wavelets. An analysis tool.} The Clarendon
Press, Oxford University Press, New York (1995).

\bibitem{FBhormander83} {\sc L. H\"{o}rmander.} \textit{The analysis of linear partial differential operators. I.} Grundlehren der Mathematischen Wissenschaften 256, Springer-Verlag, Berlin (1983).

%\bibitem{FBjaffard89} {\sc S. Jaffard.} Exposants de {H}\"{o}lder en des points donn\'{e}s et coefficients d'ondelettes. \textit{Comptes rendus de l'Acad\'emie des Sciences Series I  Mathematics} \textbf{308} (1989), no. 4, 79--81.

\bibitem{FBkula09} {\sc G. Kutyniok, D. Labate.} Resolution of the wavefront set using continuous shearlets.
\textit{Trans. Amer. Math. Soc.} \textbf{361} (2009), no. 5, 2719--2754.

\bibitem{FBkula12} {\sc G. Kutyniok, D. Labate.} \textit{Shearlets.} Appl. Numer. Harmon. Anal.
Birkh\"auser/Springer, New York (2012).

\bibitem{FBkpsv2014} {\sc S. Kostadinova, S. Pilipovi\'{c}, K. Saneva, J. Vindas.}
The ridgelet transform of distributions. \textit{Integral Transforms Spec. Funct.} \textbf{25} (2014), no. 5, 344--358.


\bibitem{FBmallat09} {\sc S. Mallat.} \textit{A wavelet tour of signal processing, The sparse way.} Elsevier, Academic Press, Amsterdam (2009).


\bibitem{FBpilvul} {\sc  S. Pilipovi\'{c}, M. Vuleti\'c.}
 Characterization of wave front sets by wavelet transforms. \textit{Tohoku Math. J.} \textbf{58} (2006), no. 3, 369--391.

\bibitem{FBPRTV}  {\sc  S. Pilipovi\'c, D. Raki\'c, N. Teofanov, J. Vindas.}
The wavelet transforms in Gelfand-Shilov spaces,
\textit{Collectanea Mathematica} \textbf{67} (2016), no. 3, 443--460.

\bibitem{FBLS1951} {\sc L. Schwartz.} \textit{Th\'eorie des distributions. Tome II.}  Actualit\'es Sci. Ind., no. 1122 Publ.
Inst. Math. Univ. Strasbourg 10. Hermann \& Cie., Paris (1951).

\bibitem{FBsomu17}
{\sc S. Sonoda, N. Murata.} Neural network with unbounded activation functions is universal approximator.
\textit{Applied and Computational Harmonic Analysis} \textbf{43} (2017), no. 2,
233--268.

\bibitem{FBTreves1967} {\sc F. Tr\`eves}. \textit{Topological vector spaces, distributions and kernels.} Academic Press, New York-London (1967).

\bibitem{FBwong02}
{\sc M.W. Wong}. \textit{Wavelet transforms and localization operators.} Operator Theory: Advances and Applications. Birkh\"{a}user Verlag, Basel (2002).

\end{thebibliography}
\end{document}